\documentclass[a4paper,12pt]{article}
\usepackage{amsmath}
\usepackage{amsthm}
\usepackage{amssymb}
\usepackage{amscd}
\usepackage{graphicx}
\usepackage{epsfig}
\usepackage[matrix,arrow,curve]{xy}
\usepackage{color}
\usepackage{mathrsfs}
\usepackage[scr=rsfs,cal=boondox]{mathalpha}

\usepackage{bbm}
\usepackage{stmaryrd}
\usepackage{hyperref}

\usepackage{latexsym}
\usepackage{amsfonts}
\input xy

\newtheorem{theorem}{Theorem}[section]
\newtheorem{proposition}[theorem]{Proposition}
\newtheorem{lemma}[theorem]{Lemma}
\newtheorem{corollary}[theorem]{Corollary}

\theoremstyle{definition}
\newtheorem{example}[theorem]{Example}
\newtheorem{remark}[theorem]{Remark}
\newtheorem{definition}[theorem]{Definition}

\font\black=cmbx10 \font\sblack=cmbx7 \font\ssblack=cmbx5 \font\blackital=cmmib10  \skewchar\blackital='177
\font\sblackital=cmmib7 \skewchar\sblackital='177 \font\ssblackital=cmmib5 \skewchar\ssblackital='177
\font\sanss=cmss11 \font\ssanss=cmss8 scaled 900 \font\sssanss=cmss8 scaled 600 \font\blackboard=msbm10
\font\sblackboard=msbm7 \font\ssblackboard=msbm5 \font\caligr=eusm10 \font\scaligr=eusm7 \font\sscaligr=eusm5

\font\bsymb=cmsy10 scaled\magstep2
\def\all#1{\setbox0=\hbox{\lower1.5pt\hbox{\bsymb
       \char"38}}\setbox1=\hbox{$_{#1}$} \box0\lower2pt\box1\;}
\def\exi#1{\setbox0=\hbox{\lower1.5pt\hbox{\bsymb \char"39}}
       \setbox1=\hbox{$_{#1}$} \box0\lower2pt\box1\;}

\def\tx#1{{\fam0\relax#1}}

\newfam\bifam
\textfont\bifam=\blackital \scriptfont\bifam=\sblackital \scriptscriptfont\bifam=\ssblackital

\newfam\blfam
\textfont\blfam=\black \scriptfont\blfam=\sblack \scriptscriptfont\blfam=\ssblack

\newfam\bbfam
\textfont\bbfam=\blackboard \scriptfont\bbfam=\sblackboard \scriptscriptfont\bbfam=\ssblackboard

\newfam\ssfam
\textfont\ssfam=\sanss \scriptfont\ssfam=\ssanss \scriptscriptfont\ssfam=\sssanss
\def\sss#1{{\fam\ssfam\relax#1}}

\newfam\clfam
\textfont\clfam=\caligr \scriptfont\clfam=\scaligr \scriptscriptfont\clfam=\sscaligr

\def\pmb#1{\setbox0\hbox{${#1}$} \copy0 \kern-\wd0 \kern.2pt \box0}
\def\pmbb#1{\setbox0\hbox{${#1}$} \copy0 \kern-\wd0
      \kern.2pt \copy0 \kern-\wd0 \kern.2pt \box0}
\def\pmbbb#1{\setbox0\hbox{${#1}$} \copy0 \kern-\wd0
      \kern.2pt \copy0 \kern-\wd0 \kern.2pt
    \copy0 \kern-\wd0 \kern.2pt \box0}
\def\pmxb#1{\setbox0\hbox{${#1}$} \copy0 \kern-\wd0
      \kern.2pt \copy0 \kern-\wd0 \kern.2pt
      \copy0 \kern-\wd0 \kern.2pt \copy0 \kern-\wd0 \kern.2pt \box0}
\def\pmxbb#1{\setbox0\hbox{${#1}$} \copy0 \kern-\wd0 \kern.2pt
      \copy0 \kern-\wd0 \kern.2pt
      \copy0 \kern-\wd0 \kern.2pt \copy0 \kern-\wd0 \kern.2pt
      \copy0 \kern-\wd0 \kern.2pt \box0}

\mathchardef\za="710B  
\mathchardef\zb="710C  
\mathchardef\zg="710D  
\mathchardef\zd="710E  
\mathchardef\zve="710F 
\mathchardef\zz="7110  
\mathchardef\zh="7111  
\mathchardef\zvy="7112 
\mathchardef\zi="7113  
\mathchardef\zk="7114  
\mathchardef\zl="7115  
\mathchardef\zm="7116  
\mathchardef\zn="7117  
\mathchardef\zx="7118  
\mathchardef\zp="7119  
\mathchardef\zr="711A  
\mathchardef\zs="711B  
\mathchardef\zt="711C  
\mathchardef\zu="711D  
\mathchardef\zvf="711E 
\mathchardef\zq="711F  
\mathchardef\zc="7120  
\mathchardef\zw="7121  
\mathchardef\ze="7122  
\mathchardef\zy="7123  
\mathchardef\zf="7124  
\mathchardef\zvr="7125 
\mathchardef\zvs="7126 
\mathchardef\zf="7127  
\mathchardef\zG="7000  
\mathchardef\zD="7001  
\mathchardef\zY="7002  
\mathchardef\zL="7003  
\mathchardef\zX="7004  
\mathchardef\zP="7005  
\mathchardef\zS="7006  
\mathchardef\zU="7007  
\mathchardef\zF="7008  
\mathchardef\zW="700A  

\newcommand{\be}{\begin{equation}}
\newcommand{\ee}{\end{equation}}

\newcommand{\bea}{\begin{eqnarray}}
\newcommand{\eea}{\end{eqnarray}}
\newcommand{\beas}{\begin{eqnarray*}}
\newcommand{\eeas}{\end{eqnarray*}}
\def\*{{\textstyle *}}
\newcommand{\R}{{\mathbb R}}
\newcommand{\T}{{\mathbb T}}

\newcommand{\we}{\wedge}
\newcommand{\nn}{\nonumber}
\newcommand{\ot}{\otimes}

\newcommand{\pa}{\partial}
\newcommand{\ti}{\times}

\newcommand{\cG}{{\mathcal G}}
\newcommand{\Li}{{\cal L}}

\def\lan{\langle}
\def\ran{\rangle}

\def\op{\oplus}

\def\cX{\mathcal{X}}
\def\cY{\mathcal{Y}}

\def\wt{\widetilde}

\def\Sec{\operatorname{Sec}}

\def\la{\langle}
\def\ran{\rangle}

\def\sT{{\sss T}}

\def\xd{\tx{d}}
\def\xi{\tx{i}}

\def\dt{\xd_{\sT}}

\def\dtr{\dt^r}

\def\cF{{\mathcal F}}

\newdir{|>}{%
!/4.5pt/@{|}*:(1,-.2)@^{>}*:(1,+.2)@_{>}}

\def\dim{\operatorname{dim}}

\def\gl{\operatorname{gl}}
\def\GL{\operatorname{GL}}

\newdir{ (}{{}*!/-5pt/@^{(}}

\newcommand{\tr}{\mbox{$\mathrm{tr}$}}

\newcommand{\id}{\mathrm{id}}
\newcommand{\inv}{\mathrm{inv}}

\newcommand{\Tor}{\textnormal{Tor}}
\newcommand{\Vt}{\textnormal{Vert}}
\newcommand{\aGamma}{\overset{\alpha}{\Gamma}}

\def\N{\mathbb{N}}

\def\Lie{\mathrm{Lie}}
\def\n{\nabla}
\def\g{\frak{g}}
\def\T{\mathscr{T}}

\begin{document}
\title{Lifting statistical structures}
\author{Katarzyna Grabowska\footnote{email:konieczn@fuw.edu.pl }\\
\textit{Faculty of Physics,
                University of Warsaw}
\\ \\
Janusz Grabowski\footnote{email: jagrab@impan.pl} \\
\textit{Institute of Mathematics, Polish Academy of Sciences}
\\ \\
Marek Ku\'s\footnote{email: marek.kus@cft.edu.pl}\\
\textit{Center for Theoretical Physics, Polish Academy of Sciences}
 \\ \\
Giuseppe Marmo\footnote{email: marmo@na.infn.it}\\
\textit{Dipartimento di Fisica ``Ettore Pancini'', Universit\`{a} ``Federico II'' di Napoli} \\
\textit{and Istituto Nazionale di Fisica Nucleare, Sezione di Napoli} }
\date{}
\maketitle
\begin{abstract}
We consider some natural (functorial) lifts of  geometric objects associated with statistical manifolds (metric tensor, dual connections, skewness tensor, etc.) to higher tangent bundles.
It turns out that the lifted objects form again a statistical manifold structure, this time on the higher tangent bundles, with the only difference that the metric tensor is pseudo-Riemannian.
What is more, natural lifts of potentials (called also divergence or contrast functions) turn out to be again potentials, this time for the lifted statistical structures. We propose an analogous procedure for lifting statistical structures on Lie algebroids and lifting contrast functions which are defined on Lie groupoids. In particular, we study in detail Lie groupoid structures of higher tangent bundles of Lie groupoids. Our geometric constructions of lifts are illustrated by explicit examples, including some important statistical models and potential functions on Lie groupoids.

\par


\end{abstract}
\section{Introduction}
Since the pioneering work of Amari and Chentsov \cite{Amari1982,Amari1985,Amari1987,Amari2007,Amari2012,Amari2016,Censov1982}, information geometry has flourished greatly and became an object of intensive studies and various applications, especially in statistical decision rules and optimal inferences.
This theory appeared to provide a link  between different disciplines where statistical-probabilistic aspects play a relevant role. The search for such a link was clearly advocated by J.~A.~Wheeler \cite{Wheeler1989}. A remarkable theorem by Chentsov \cite{Censov1982}, proved in a categorical setting, states that the Fisher-Rao metric tensor is the only one respecting a monotonicity requirement under coarse grained transformations and, moreover, that the metric is invariant under the diffeomorphism group of the sample space.

Also quantum information geometry has been developed, which is not strange, as the standard quantum mechanics is a probabilistic-statistical theory.
The geometrical formulation of quantum mechanics has been used to show that the Fisher-Rao metric tensor can be obtained as the pull-back of the Fubini-Study metric tensor, when the statistical manifold is embedded as an appropriate Lagrangian submanifold of the complex projective space of pure quantum states \cite{Facchi2010}. The possibility to obtain the Fisher-Rao metric from the Fubini-Study metric shows a deep connection between geometric information theory and the geometrical formulation of quantum mechanics (indeed, both theories admit a statistical-probabilistic description).

The geometrical objects  introduced for statistics are Riemannian metrics, symmetric $(0,3)$-tensors, and dual pairs of torsionless affine connections. Such geometric structures give rise to the concept of a so called \emph{statistical manifold}, which is a purely geometric abstract of the geometries of statistical models. Moreover, the notion of `directed distances' (called also distinguishability functions, divergence functions, or contrast functions), introduced as potential functions in the classical setting, serve as `generating objects' of statistical manifolds. Classical potentials are related to relative entropies in the quantum setting. Thus, the latter were used as potential functions for quantum metrics on the space of quantum states. It should be noticed, however, that potential functions appear also in the description of K\"ahler manifolds, Frobenius manifolds, and that the connection with statistical manifolds was considered in a paper by Y.~Manin and his collaborators \cite{Combe2021}.

A natural problem is to find natural examples of statistical structures. An obvious idea is to build additional statistical models out of a given one. One possibility is lifting statistical structures to fibrations over statistical manifolds. In this paper we consider this lifting problem from the point of view of intrinsic differential geometry and functorial  constructions, by lifting statistical structures to higher tangent bundles. We will show that this approach provides lifted potential functions for the lifted geometric objects, an  interesting result providing an effective way of producing whole families of statistical structures.

The `naturality requirement' for the construction provides metric tensors which are not Riemannian but only pseudo-Riemannian, that is nowadays commonly accepted  in the literature, if statistical structures not necessarily associated with statistical models are concerned. As the metric is used mainly to introduce a notion of \emph{distinguishability} among probability distributions or quantum states, this rises the question of a proper interpretation for probability distributions which have zero-distance even though they are  distinct.
This aspect, however, will not be considered in this paper and will be dealt with in the future.

As we observed in \cite{Grabowska2019} (see also \cite{Grabowska2020}), all the methods of inducing statistical structures out of contrast functions actually depend only on the Lie groupoid structure of the pair groupoid $G=M\ti M$, and can be naturally generalized to the case of information geometry on Lie groupoids and Lie algebroids. Contrast functions on Lie groupoids determine Lie algebroid tensors and connections by differentiating the contrast functions with respect to the left- and the right-invariant vector fields associated with Lie algebroid sections.

It is known that higher tangent bundles of Lie groupoids (resp., Lie algebroids) are canonically also Lie groupoids (resp., Lie algebroids). We include into the paper deeper studies on Lie groupoid and Lie algebroid structures on higher tangent bundles of Lie groupoids and Lie algebroids. Then, we define higher lifts of statistical Lie algebroids. We also prove that, in the case when the original statistical structure on a Lie algebroid comes from a contrast function on the corresponding (local) Lie groupoid, the lifted statistical structure comes from the lift  of the contrast function. Additionally, we offer numerous examples of all these procedures, including
some important statistical models.

\section{Differential geometry of statistical models}

Let us consider a paradigmatic experiment in which we perform a series of measurements obtaining some experimental distribution of them. Typically, upon our results we want to find the `true' distribution, or rather identify a distribution (taken from a prescribed family), that fits best to the obtained data. It means that we are dealing with probability distributions $ p(\omega,\mathbf{x})d\omega $ depending on some parameters $\mathbf{x}$, identifying members of the considered family of distributions. Thus, to achieve the mentioned goals we need a parameterization of the family of distributions and possibly a kind of distance-like function that measures a relative distance between distributions, or more generally, allows distinguishing two of them.

\textit{Information geometry} provides a link between statistical models and differential (Riemannian) geometry. The idea is to parameterize the space of probability distributions on a measure space $ \Omega $ (the sample space) with a measure $\xd\omega $.
Let thus  $ \mathcal{P}(\Omega) $ denote such a  family of probability distributions in the sample space.  We parameterize it \textit{via} a map $ M \ni x \mapsto p(\omega,x)\xd\omega $ from a differentiable manifold  $ M $. That is, for each point $ x $ in the parameter space $ M $, we have a probability measure $ p(\cdot, x) $ on $ \Omega$.

The notions of a distance and distinguishability  between distributions  can be formulated \cite{Amari2007, Amari2016} by introducing a two-point \emph{potential function} $F: M \times  M \rightarrow \mathbb{R}$  (see e.g. \cite{Ciaglia2017} for an intrinsic approach) usually called a \emph{contrast function} or a \emph{divergence}.  It is a function vanishing on the diagonal $\zD_M=\{ (x,x)\,|\, x\in M\}$ and such that its Hessians at points of the diagonal (they have an unambiguous geometric meaning)  are non-degenerate. It is usually assumed that the potential function is non-negative, so that the Hessians are positive definite.

Let, $(x^j)$ be a coordinate system on the first  copy of $ M $ and $(y^j)$ be the same coordinates but on the second copy of $M$. If $F$ is at least $C^3$, the condition imposed on $F$ implies \cite{Matumoto1993}
$$
	\left. \frac{\partial F}{\partial x^j}\right|_{x=y}=\left. \frac{\partial F}{\partial y^j}\right|_{x=y}=0\,.
$$
We can now construct a two- and a three-tensor $g$ (a metric) and $T$ (the skewness tensor),
$$
	g_{jk}=\left. \frac{\partial^2F}{\partial x^j\partial x^k}\right|_{x=y}=\left. \frac{\partial^2F}{\partial y^j\partial y^k}\right|_{x=y}=-\left. \frac{\partial^2F}{\partial y^j\partial x^k}\right|_{x=y},
$$
and
$$
	T_{jkl}=\left. \frac{\partial^3F}{\partial x^l\partial y^k\partial y^j}\right|_{x=y}=-\left. \frac{\partial^3F}{\partial y^l\partial x^k\partial x^j}\right|_{x=y}.
$$
For given tensors $g$ and $T$, we consider a one-parameter family of torsionless connections
$$
	\Gamma^\alpha_{jkl}:=\Gamma^{g}_{jkl}-\frac{\alpha}{2}T_{jkl}\,,
$$
where $\Gamma^g_{jkl}$ are the Christoffel symbols for the metric tensor $g$ (the Levi-Civita connection). For every Riemannian manifold there is a concept of  \emph{duality} of connections, namely $\nabla$ and $\nabla^\ast$ are dual with respect to $g$ if for all vector fields $X,Y,Z$ on $ M $
\be\label{eq:dual}
	Z\left(g(X,Y)\right)= g\left(\nabla_ZX,Y\right) +g\left(X, \nabla^{\ast}_Z Y \right).
\ee
One can check that $\nabla^\alpha$ and $\nabla^{-\alpha}$ are dual to each other, i.e. $(\nabla^\alpha)^\ast=\nabla^{-\alpha}$.
The statistical model is called \textit{self-dual} if $\nabla^\alpha=\nabla^{-\alpha}$.
The self-duality  implies $T=0$.

The most prominent example of a statistical model is that proposed by Rao \cite{Rao1945}. The manifold $ M $ is equipped with the metric (the Fisher-Rao metric)
$$
	g_{jk}(x)  = \int_\Omega
	p(\omega,x)
	\left( {\frac{{\partial \log \left(p(\omega,x) \right)}}{{\partial x ^j }}}\right)
	\left( {\frac{{\partial \log \left(p(\omega,x) \right)}}{{\partial x ^k }}} \right)d\omega
$$
and the corresponding skewness tensor
$$
	T_{jkl}(x)  = \int_\Omega
	p(\omega,x)
	\left( {\frac{{\partial \log \left(p(\omega,x) \right)}}{{\partial x ^j }}}\right)
	\left( {\frac{{\partial \log \left(p(\omega,x) \right)}}{{\partial x ^k }}} \right)\left( {\frac{{\partial \log \left(p(\omega,x) \right)}}{{\partial x^l }}} \right)d\omega\,.
$$
As a contrast function defining the Fisher-Rao metric we can take the so called \emph{Kullback-Leibler divergence} (\emph{Shannon relative entropy})
\be\label{KL}
F(x,y)=	 D_{KL}(x,y)=\int_\Omega p(\omega,x)\log\frac{p(\omega,x)}{p(\omega,y)}d\omega\,.
\ee
The geometrical structures of statistical models \cite{Amari1982,Amari1985,Amari1987,Amari2007,Amari2012,Amari2016,Barndorff1986,Censov1982,Murray1993,Rao1945} were an inspiration for the concept of a general \emph{statistical manifold} \cite{Lauritzen1987}, which is \emph{a priori} not necessarily associated with a statistical model. It can be analyzed in terms of (pseudo)-Riemannian geometry without any particular connection to statistics. However, every statistical manifold has a statistical model, as it was demonstrated in \cite{Le2006}. In the following section we shall review the definition of a statistical manifold as a geometric construction.

\section{Statistical manifolds}
The original definition of a statistical manifold by Lauritzen \cite{Lauritzen1987} is the following (see also \cite{Calin2014,Matsuzoe2007,Matumoto1993,Zhang2020}).
\begin{definition}
A \emph{statistical manifold} is a Riemannian manifold $(M,g)$ with a symmetric covariant 3-tensor $T$.
\end{definition}
\begin{remark}\label{re}
It is obvious that any submanifold $N$ of a statistical manifold is a statistical manifold itself with the tensors $g_N=g\big|_N$ and $T_N=T\big|_N$.
\end{remark}
\noindent All geometric statistical models carry such a structure. The tensor $T$ is called the \emph{skewness tensor} (in statistical models -- \emph{Amari-Chensov tensor}) or a \emph{cubic form} (\emph{cubic tensor}). Recently, also pseudo-Riemannian metrics have been admitted in the definition of statistical manifolds, as all main features of statistical manifolds may be carried over to this more general framework. We will use this more general definition. Note however, that Remark \ref{re} is no longer true in the pseudo-Riemannian case.

\bigskip\noindent
In our considerations, all geometric objects (manifolds, functions, vector fields, metrics, etc.) will be smooth and all connections will be affine connections.

There are several equivalent definitions of a statistical manifold in the literature. Depending on the problem, one or another may be more suitable. A particular r\^ole is played  by affine connections and their duals with respect to the metric. We mentioned this already in the previous section about statistical models, see (\ref{eq:dual}).
\begin{definition}
Let $\nabla$ be a connection on a pseudo-Riemannian manifold $(M,g)$. The \emph{dual connection} (with respect to $g$), called also the \emph{$g$-conjugate connection}, is the connection $\nabla^*$ defined by
$$g(\nabla^*_XY,Z)=X\,g(Y,Z)-g(Y,\nabla_XZ)\,.$$
\end{definition}
\noindent Note that it is a true duality as $\left(\nabla^*\right)^*=\nabla$, and that the Levi-Civita connection $\nabla^g$ for $(M,g)$ is the only self-dual connection which is torsion-free.

\medskip\noindent
We will say that two geometric structures on a pseudo-Riemannian manifold are \emph{equivalent} if one of them canonically determines the other and \emph{vice versa}.
The following is essentially due to Lauritzen \cite{Lauritzen1987}.
\begin{theorem}\label{statm} Let $(M,g)$ be a pseudo-Riemannian manifold. The following geometric structures on $(M,g)$ are equivalent.
\begin{enumerate}
\item a totaly symmetric $(0,3)$-tensor $T$;
\item a torsion-free connection $\nabla$ such that $\nabla g$ is symmetric ($(g,\nabla)$ are \emph{Codazzi coupled});
\item a torsion free connection $\nabla$ such that $$\nabla^g=\frac{1}{2}(\nabla+\nabla^*)\,;$$
\item a pair $(\nabla,\nabla^*)$ of dual torsion-free connections (\emph{dualistic structure}).
\end{enumerate}
The tensor $T$ and the connections $\nabla,\nabla^*$ are related by
\beas g\left(\nabla_XY,Z \right)&=&g\left(\nabla^g_XY,Z \right)-\frac{1}{2}T(X,Y,Z)\,,\\
g\left(\nabla^*_XY,Z \right)&=&g\left(\nabla^g_XY,Z \right)+\frac{1}{2}T(X,Y,Z)\,.
\eeas
\end{theorem}
\begin{remark}
It is clear that $(M,g,\nabla)$ is a statistical manifold if and only if $(M,g,\nabla^*)$ is a statistical manifold, the so called \emph{dual statistical manifold}.
Actually, the dual pair $(\nabla,\nabla^*)$ can be extended to a one-parametr family of connections,
$$g\left(\nabla^\za_XY,Z \right)=g\left(\nabla^g_XY,Z \right)-\frac{\za}{2}T(X,Y,Z)\,.$$
Here, $\za\in\R$ and $\nabla^{-\za}=\left(\nabla^\za\right)^*$. Clearly $\nabla^1=\nabla$ (called the \emph{exponential connection}), $\nabla^{-1}=\nabla^*$ (called the \emph{mixture connection}), and $\nabla^0=\nabla^g$ is the metric (Levi-Civita) connection. The tensor $T$ can be obtained from the connection \emph{via} the formula
$$T(X,Y,Z)=g\left(\nabla^*_XY,Z \right)-g\left(\nabla_XY,Z \right)=(\nabla_X\,g)(Y,Z)$$
and is sometimes also called the \emph{cubic form of $(M,g,\nabla)$}.
\end{remark}
\noindent Recently, \emph{statistical manifolds admitting torsion} are considered as well \cite{Henmi2011,Kurose2007,Matsuzoe2010}.
\begin{definition} Let $g$ be a pseudo-Riemannian metric on a manifold $M$ and $\nabla$ be a connection on $M$.
We call the pair $(M,g,\nabla)$ a \emph{statistical manifold admitting torsion} (SMAT) if
$$(\nabla_Xg)(Y,Z)-(\nabla_Yg)(X,Z)=-g(\Tor^\nabla(X,Y),Z)\,.$$
In this case we call the pair $(g,\nabla)$ \emph{torsion coupled}.
\end{definition}
\begin{proposition} A pair $(g,\nabla)$ is torsion coupled if and only if\, $\nabla^*$ is torsion-free.
\end{proposition}
\noindent In consequence, if $\nabla,\nabla^*$ are torsion coupled with $g$, then $(M,g,\nabla,\nabla^*)$ is a statistical manifold.

\section{Contrast functions}
A very useful method of constructing statistical manifolds is that by means of contrast functions \cite{Blesild1991,Ciaglia2017,Ciaglia2018,Ciaglia2019,Euguchi1985,Euguchi1992}, which are also called \emph{potentials, divergences, yokes}, etc.
Let $M$ be a manifold and $X$ be a vector field on $M$. Denote with $^1\!X$ (resp., $^2\!X$) the vector field on $M\ti M$ which is $X$ on the left factor (resp., on the right factor). In coordinates, if $(z^i)$ are coordinates on $M$ and $(x^i)$, $(y^i)$ denote the same coordinates on the first and the second factor, we have
$$^1(f(z)\pa_{z^i})(x,y)=f(x)\pa_{x^i}\,,\quad ^2(f(z)\pa_{z^i})(x,y)=f(y)\pa_{y^i}\,.$$
Note that the vector fields $^1\!X$ and $^2Y$ commute for all $X,Y$.
It will be convenient to use the notation
$$^1\!X_1\cdots ^1\!\!X_k\,^2Y_1\cdots ^2Y_l(F)=:F[X_1\cdots X_k|Y_1\cdots Y_l]\,.$$

\begin{remark}\label{rem} For a function $F$ on a manifold $N$ of dimension $n$, with vanishing $k$-th jet at $p\in N$, and for any vector fields $X_1,\dots,X_{k+1}$ on $N$, the expression $X_1\cdots X_{k+1}(F)(p)$ depends only on $X_i(p)$, $i=1,\dots,k+1$, and
\be\label{Fmetric} g^F(p)(X_1,\dots,X_{k+1})=(-1)^k\,(X_1\cdots X_{k+1})(F)(p)\ee
defines uniquely a symmetric $(0,k+1)$-tensor $g^F(p)$ from $(\sT^*_p)^{\ot k}N$ (cf. \cite[Lemma 5.1]{Grabowska2019}).
We say that $g^F(p)$ is \emph{induced} by the function $F$.

More generally, if $N_0$ is a closed submanifold of $N$ and a function $F:N\to\R$ has vanishing $k$-th jets at points of $N_0$, then for points $p\in N_0$ the expression $X_1\cdots X_{k+1}(F)(p)$ depends only on the classes $[X_i(p)]$ of $X_i(p)$ in the normal bundle
$$E(N_0,N)=(\sT N)\big|_{N_0}/\sT N_0$$
of $N_0\subset N$, $i=1,\dots,k+1$, and
$$ g^F([X_1],\dots,[X_{k+1}])=(-1)^k\,(X_1\cdots X_{k+1})(F)\,\big|_{N_0}$$
defines uniquely a symmetric $(k+1)$-tensor $g^F$ being a section of  $E^*(N_0,N)^{\ot k}$, where $E^*(N_0,N)$ is the vector bundle dual to the normal bundle $E(N_0,N)$.

\begin{definition}\label{Ncf} We say that $g^F$ is \emph{induced} by $F$, and that $F:N\to\R$ is a \emph{contrast function on $(N_0,N)$} if $g^F$ is non-degenerate (it is a `pseudo-Riemannian metric' on $E(N_0,N)$).
\end{definition}
\end{remark}
\noindent The following proposition will be useful while discussing local forms of contrast functions and other geometric objects on statistical manifold.
\begin{proposition}\label{kjet} If $u^i$ are local coordinates in a neighbourhood of $p\in N$, $u^i(p)=0$, then $F:N\to\R$ has the $k$-th jet at $p$ vanishing if and only if $F$ can be locally written as
\be\label{kjet1} F(u)=F_{\za_1,\dots,\za_n}(u)\,(u^1)^{\za_1}\cdots (u^n)^{\za_{n}}, \quad \alpha_j\in \N\,, \quad \sum_{j=1}^{n}\alpha_j=k+1\,,
\ee
for some functions $F_{\za_1,\dots,\za_{n}}$ defined in a neighbourhood of $p$. Moreover, for the $(0,k+1)$-tensor (\ref{Fmetric}) we have
$$g^F\left(\pa_{u^{i_1}},\dots,\pa_{u^{i_{k+1}}}\right)(p)=(\zb_1)!\cdots(\zb_n)!\,F_{\zb_1,\dots,\zb_n}(p)\,,$$
where $\zb_j$ is the number of those $i_l$ which equal $j$.
\end{proposition}
\begin{proof} Indeed, it is well known (cf. also the proof of the next proposition) that if $F(p)=0$, then there exist functions $F_i$ such that $F(u)=F_i(u)\,u^i$. We get (\ref{kjet}) by induction. The rest follows easily by differentiating (\ref{kjet1}). The converse is obvious.

\end{proof}
\noindent  A particular case of contrast functions used in the theory of statistical manifolds is the following. Denote the diagonal submanifold in $M\ti M$, i.e. $\{(x,x)\,|\, x\in M\}\subset M\ti M$, with  $\zD_M$.
\begin{definition}
A function $F:M\ti M\to \R$ we call a \emph{contrast function} (\emph{potential function}, \emph{divergence function}) on $M$ if the first jets of $F$ vanish on $\zD_M$ (i.e. $F\big|_{\zD_M}=0$ and $\xd F\big|_{\zD_M}=0$) and
$g^F$ is a pseudo-Riemannian metric.
\end{definition}
\begin{remark}
Note that, in consequence of the fact that we admit pseudo-Riemannian metrics in the definition of a statistical manifold, our concept of a contrast function is more general that the `classical' one and accepts $g^F$ to be only pseudo-Riemannian.
\end{remark}
\begin{remark} According to Remark \ref{rem}, the pseudo-Riemannian metric $g^F$ reads
$$g^F(X,Y)(x)=-F[X|Y](x,x)\,.$$
\end{remark}
\noindent Of course, if $F$ is a contrast function and $F\ge 0$, then $g$ is Riemannian. Denote $F^*(x,y)=F(y,x)$. It is easy to see that $F$ is a contrast function if and only if $F^*$ is a contrast function and that $g^F=g^{F^*}$.
\begin{remark}
Actually, any statistical manifold is induced by a contrast function \cite{Matumoto1993}. It is interesting that contrast functions on $M$ define also symplectic structures on $M\ti M$ \cite{Barndorff1997}.
\end{remark}
\noindent
Let $M$ be a manifold of dimension $n$, let $p\in M$, and $(x^i)$ be local coordinates on $M$ in a neighbourhood of $p$, $x^i(p)=0$. Let $(y^i)$ be the same coordinates on another copy of $M$, so that $(x^i,y^j)$ are local coordinates in a neighbourhood of $(p,p)\in M\ti M$.
\begin{theorem}\label{thcontrast} {\rm (Local characterization of contrast functions)} Let $M$ be a manifold, $p\in M$, and $(x^i,y^i)$ be the local coordinates in a neighbourhood of $(p,p)$ as above. Then a function  $F:M\ti M\to\R$ is a contrast function on $M$ with the induced pseudo-Riemannian metric $g^F$, which locally reads
$$ g^F=g^F_{ij}(x)\,\xd x^i\ot\xd x^j\,,$$
if and only if in a neighbourhood of $(p,p)=(0,0)$ we have
\be\label{contrast}
F(x,y)=\frac{1}{2}\,(x^i-y^i)(x^j-y^j)\,h_{ij}(x,y)\,,
\ee
where $h_{ij}=h_{ji}$ and $[h_{ij}(0,0)]$ is an invertible matrix. In such a case, $g^F_{ij}(x)=h_{ij}(x,x)$ for $x$ in a neighbourhood of $0$.
\end{theorem}
\begin{proof} Suppose $F$ is a contrast function on $M$, and consider new coordinates in a neighbourhood of $(p,p)\in M\ti M$:
$$u^i=x^i-y^i\,\quad v^i=x^i+y^i\,.$$
The diagonal submanifold $\zD_M$ is defined locally by $u^i=0$, $i+1,\dots,n$, and $F(0,v)=0$. Moreover, as $\xd F(0,v)=0$, we have
$$\left(\frac{\pa F}{\pa u^j}\right)(0,v)=0\quad\text{for}\quad j=1,\dots,n\,.$$
For fixed $(u,v)$ and  $t\in\R$ put $G(t)=F(tu,v)$. We have
$$F(u,v)=G(1)-G(0)=\int_0^1G'(s)\,\xd s=\int_0^1\frac{\pa F}{\pa u^i}(su,v)\,u^i\,\xd s=\left(\int_0^1\frac{\pa F}{\pa u^i}(su,v)\,\xd s\right) u^i\,.$$
Denote
$$h_i(u,v)=\left(\int_0^1\frac{\pa F}{\pa u^i}(su,v)\,\xd s\right)\,,$$
so that $F(u,v)=h_i(u,v)u^i$. Since $\left(\frac{\pa F}{\pa u^j}\right)(0,v)=h_i(0,v)=0$, repeating the previous
calculation we get $h_i(u,v)=\frac{1}{2}\,h_{ij}(u,v)\,u^j$, so finally
$$F(u,v)=\frac{1}{2}\,u^i\,u^j\,h_{ij}(u,v)$$
for some functions $h_{ij}=h_{ji}$ defined in a neighbourhood of $(p,p)$.
It is now easy to see that
\be\label{metric00} g^F_{ij}(0)=g^F(\pa_{x^i},\pa_{x^j})(x)=-\frac{\pa^2 F}{\pa x^i\pa y^j}(x,x)=h_{ij}(x,x)\,.\ee
Conversely, if $F$ has the local form (\ref{contrast}) with $h_{ij}=h_{ji}$, then clearly $F(x,x)=0$. Moreover,
$$\frac{\pa F}{\pa x^i}(x,x)=\big[(x^j-y^j)h_{ij}\big](x,x)=0\,.$$
Similarly, $\frac{\pa F}{\pa y^i}(x,x)=0$, so $\xd F(x,x)=0$.
Finally, we show, exactly like in (\ref{metric00}), that the induced tensor $g^F$ satisfies $g^F_{ij}(x)=h_{ij}(x,x)$. Hence, $g^F$ is non-degenerate, so a pseudo-Riemannian metric, that finishes the proof.

\end{proof}
\noindent  A straightforward generalization of the above result, with completely analogous proof, is the following.
\begin{theorem}\label{thcontrast1}
Let $N_0$ be a closed submanifold of a manifold $N$ and $F:N\to\R$. Then $F$ is a contrast function on $(N_0,N)$ if and only if, for any local coordinates $(\zx^a,z^i)$ in $N$, in which $N_0$ is defined by $z^i=0$, the function $F$ takes the form
\be\label{contrast1}
F(\zx,z)=\frac{1}{2}\,z^iz^j\,h_{ij}(\zx,z)\,,
\ee
where $h_{ij}=h_{ji}$ and $[h_{ij}(\zx,0)]$ are invertible matrices for all $\zx$. In this case
the `pseudo-Riemannian metric' on the vector bundle $E(N_0,N)$ reads
$$g^F(\zx)=h_{ij}(\zx,0)\,[\xd z^i]\ot[\xd z^j]\,,$$
where $[\xd z^i]$ is the section of the vector  bundle $E^*(N_0,N)$, dual to the normal bundle $E(N_0,N)$, represented by $\xd z^i$.
\end{theorem}
\noindent A more detailed local description of `classical' contrast functions is the following.
\begin{theorem} In the coordinates $(x^i,y^j)$ on $M\ti M$, any contrast function $F$ in a neighbourhood of $(p,p)$ can be written as
\bea\label{lfc} &F(x,y)=\frac{1}{2}\,(x^i-y^i)(x^j-y^j)\,g^F_{ij}(0)+x^i\,x^j\,y^k\,\zg_{ijk}(x,y)+
y^i\,y^j\,x^k\,\zvy_{ijk}(x,y)\\
&+x^i\,x^j\,x^k\,a_{ijk}(x,y)+y^i\,y^j\,y^k\,b_{ijk}(x,y)\,,\nn
\eea
where $\zg_{ijk}$ and $\zvy_{ijk}$ are symmetric with respect to the first two indices. Moreover, $t_{ijk}=\zg_{ijk}-\zvy_{ijk}$ and $a_{ijk},b_{ijk}$ are totally symmetric.
\end{theorem}
\begin{proof}
 Since $h_{ij}-h_{ij}(p,p)$ vanishes at $(p,p)$, using original coordinates $(x^i, y^j)$ we can write (cf. (\ref{kjet}))
$$h_{ij}(x,y)=g^F_{ij}(0,0)+2\,\za_{ijk}(x,y)\,x^k+2\,\zb_{ijk}(x,y)\,y^k\,,$$
so that (cf. (\ref{contrast}))
$$F(x,y)=(x^i-y^i)(x^j-y^j)\left(\frac{1}{2}\,g^F_{ij}(0,0)+\za_{ijk}(x,y)\,x^k
+\zb_{ijk}(x,y)\,y^k\right)\,.$$
Since $h_{ij}=h_{ji}$, we have $\za_{ijk}=\za_{jik}$ and $\zb_{ijk}=\zb_{jik}$,
so finally
\beas &F(x,y)=\frac{1}{2}\,(x^i-y^i)(x^j-y^j)\,g^F_{ij}(0)+x^i\,x^j\,y^k\,\zg_{ijk}(x,y)+
y^i\,y^j\,x^k\,\zvy_{ijk}(x,y)\\
&+x^i\,x^j\,x^k\,a_{ijk}(x,y)+y^i\,y^j\,y^k\,b_{ijk}(x,y)\,,
\eeas
where
$$\zg_{ijk}=\zb_{ijk}-\left(\za_{ikj}+\za_{jki}\right)\,,\quad \zvy_{ijk}=\za_{ijk}-\left(\zb_{ikj}+\zb_{jki}\right)\,,$$
and
$$a_{ijk}=\za_{ijk}+(cycl)\,,\quad b_{ijk}=\zb_{ijk}+(cycl)\,,
$$
where `(cycl)' denotes the cyclic permutations of $(ijk)$, and (\ref{lfc}) follows. It is obvious that
$\zg_{ijk},\zvy_{ijk}$ are symmetric with respect to the first two indices and that $a_{ijk},b_{ijk}$ are totally symmetric. Moreover,
$$t_{ijk}=\zvy_{ijk}-\zg_{ijk}=\left(\za_{ikj}+\za_{jki}+\za_{ijk}\right)-\left(\zb_{ijk}+\zb_{ikj}
+\zb_{jki}\right)$$
is totally symmetric.

\end{proof}
\begin{remark}
As follows from the proof, the functions $\zg_{ijk},\zvy_{ijk},a_{ijk},b_{ijk}$ in (\ref{lfc}) cannot be arbitrary.
\end{remark}
\noindent
It is easy to see now that
\bea &F^*(x,y)=\frac{1}{2}\,(x^i-y^i)(x^j-y^j)\,g^F_{ij}(0)+x^i\,x^j\,y^k\cdot\zvy_{ijk}(y,x)+
y^i\,y^j\,x^k\cdot\zg_{ijk}(y,x)\nonumber\\
&+x^i\,x^j\,x^k\cdot b_{ijk}(y,x)+y^i\,y^j\,y^k\cdot a_{ijk}(y,x)\,,\label{lfc1}
\eea
so that $(F-F^*)(x,y)$ reads
\bea
&(F-F^*)(x,y)=x^i\,x^j\,y^k\cdot\left(\zg_{ijk}(x,y)-\zvy_{ijk}(y,x)\right) \label{lfc2}\\
&+y^i\,y^j\,x^k\cdot\left(\zvy_{ijk}(x,y)-\zg_{ijk}(y,x)\right)
+x^i\,x^j\,x^k\cdot\left(a_{ijk}(x,y)-b_{ijk}(y,x)\right) \nonumber\\
&+y^i\,y^j\,y^k\cdot\left(b_{ijk}(x,y)-a_{ijk}(y,x)\right)\,.\nonumber
\eea
\noindent
We have already discussed a relation between the contrast function and the metric on a statistical manifold. Now, we pass to the rest of the structure. Let $F$ be a contrast function on a manifold $M$, and $g^F$ be the pseudo-Riemannian metric induced by $F$. The induced connection $\nabla^F$ and the skewness tensor $T^F$ are determined by the formulae
$$
g^F(\nabla^F_X\,Y,Z)(x)=-F[XY|Z](x,x)\,,\quad T^F(X,Y,Z)(x)=(F-F^*)[X|YZ](x,x)\,.
$$
Note that the formula for $T^F$ makes sense, since the second jets of $F-F^*$ vanish on the diagonal (cf. Remark \ref{rem}). The connection $\left(\nabla^F\right)^*$ dual to $\nabla^F$ can be obtained from
$$g^F(\left(\nabla^F\right)^*_X\,Y,Z)(x)=-F[Z|XY](x,x)=-F^*[XY|Z](x,x)\,.$$
It is easy to see that $\left(\nabla^F\right)^*=\nabla^{F^*}$.
\begin{theorem}
The connections $\nabla^F$ and $\left(\nabla^{F}\right)^*=\nabla^{F^*}$ are torsionless, and the pair $(g^F,\nabla)$ is Codazzi coupled.
\end{theorem}
\begin{corollary}
If $F$ is a contrast function on $M$, then $(M,g^F,T^F)$, $(M,g^F,\nabla^F)$, and $(M,g^F,\nabla^F, \left(\nabla^F\right)^*)$
are different presentations of the same statistical manifold.
\end{corollary}

\medskip\noindent Writing $\nabla^F$, $\left(\nabla^F\right)^*=\nabla^{F^*}\!\!$, and $T^F$ in local coordinates of Theorem \ref{thcontrast}, we get

$$
\nabla^F_{\pa_{x^i}}\pa_{x^j}=(\zG^F)_{ij}^l(x)\pa_{x^l}\,,\quad \left(\nabla^F\right)^*_{\pa_{x^i}}\pa_{x^j}=(\zG^{F^*})_{ij}^l(x)\pa_{x^l}\,,\quad T^F(\pa_{x^i},\pa_{x^j},\pa_{x^k})=T_{ijk}(x)\,,
$$

\medskip\noindent
so using the local form (\ref{lfc}) of $F$, the local form (\ref{lfc1}) of $F^*$, and the local form (\ref{lfc2}) of $(F-F^*)$, we get the following.
\begin{theorem} The local coefficients of $\nabla^F$ and $T^F$ read
\beas\label{T}
(\zG^F)^l_{ij}(x)&=&g^F\left(\nabla^F_{\pa_{x^i}}\pa_{x^j},\pa_{x^k}\right)(x)\cdot (g^F)^{kl}(x)
=-\big[\pa_{x^i}\,\pa_{x^j}\,\pa_{y^k}(F)(x,x)\big]\cdot(g^F)^{kl}(x)\\
&=&-\zg_{ijk}(x,x)\cdot (g^F)^{kl}(x);\\                      \\
(\zG^{F^*})^l_{ij}(x)&=&g^F\left(\nabla^{F^*}_{\pa_{x^i}}\pa_{x^j},\pa_{x^k}\right)(x)
\cdot(g^{F})^{kl}(x)=-[\pa_{x^i}\,\pa_{x^j}\,\pa_{y^k}(F^*)(x,x)]\cdot(g^F)^{kl}(x)\\
&=&-\zvy_{ijk}(x,x)\cdot(g^F)^{kl}(x);\\                      \\
T^F_{ijk}(x)&=&\pa_{x^i}\,\pa_{y^j}\,\pa_{y^k}(F-F^*)(x,x)=\zvy_{kji}(x,x)-\zg_{kji}(x,x)
=t_{ijk}(x,x)\,.
\eeas
\end{theorem}
\noindent It is easy to see that $\nabla^F$ is symmetric (torsionless) and $T^F$ is totally symmetric.
\begin{remark}
Statistical manifolds admitting torsion are induced from so called \emph{pre-contrast functions} \cite{Henmi2011a,Matsuzoe2010}, but we will not consider pre-contrast functions in this paper.
\end{remark}

\section{Lifting geometrical structures to the tangent bundle}\label{sec:4}
Let $M$ be a manifold, and
$$\T(M)=\bigoplus_{p,q= 0}^\infty\T^q_p(M)$$
be the algebra of tensor fields on $M$, elements of $\T ^q_p(M)$ being the $q$-contravariant and $p$-covariant tensor fields on $M$. There is a canonical injective homomorphism of the algebra $\T(M)$ into into the algebra $\T(\sT M)$ of tensor fields on the tangent bundle $\sT M$,
$$
\Vt:\T(M)\to \T(\sT M)\,,\quad K\mapsto K^v\,.
$$
The tensor $K^v$ is called the \emph{vertical lift} of $K$. Since $\Vt$ is an algebra homomorphism,
$(K\ot S)^v=K^v\ot S^v$, it is enough to define the vertical lifts of functions, 1-forms, and vector fields.
In local coordinates $(x^i)$ on $M$ and the adapted coordinates $(x^i,\dot x^j)$ on $\sT M$ the lifts read
\beas f^v(x,\dot x)&=&f(x)\,,\\
(f_i(x)\xd x^i)^v&=&f_i(x)\xd x^i\,,\\
(X_i(x)\,\pa_{x^i})^v&=&X_i(x)\pa_{\dot x^i}\,.
\eeas
There is another lift of tensor fields to the tangent bundle \cite{Grabowski1995,Yano1967,Yano1973}
$$
\dt:\T(M)\to \T(\sT M)\,,\quad K\mapsto K^c\,,
$$
called the \emph{tangent} or \emph{complete lift}. In this case, $\dt$ is a $\Vt$-derivation, i.e.
\be\label{der}\left(K\ot S\right)^c=K^c\ot S^v+K^v\ot S^c\,.\ee
Again, due to (\ref{der}), it is enough to define the complete lifts of functions, 1-forms, and vector fields:
\beas f^c(x,\dot x)&=&\frac{\pa f}{\pa x^i}(x)\,\dot x^i,\\
 \left(f_i(x)\,\xd x^i\right)^c&=& \frac{\pa f_i}{\pa x^j}(x)\,\dot x^j\,\xd x^i+f_i(x)\,\xd \dot x^i\,,\\
\left(X_i(x)\,\pa_{x^i}\right)^c&=&\frac{\pa X_i}{\pa x^j}(x)\,\dot x^j\,\pa_{\dot x^i}+X_i(x)\,\pa_{x^i}\,.
\eeas
In particular, the complete lift of a 2-covariant tensor $g=g_{ij}(x)\,\xd x^i\ot\xd x^j$ is
\be\label{g-lift}
\left(g_{ij}(x)\,\xd x^i\ot\xd x^j\right)^c=\frac{\pa g_{ij}}{\pa x^k}(x)\, \dot x^k\, \xd x^i\ot\xd x^j+g_{ij}(x)\left(\xd\dot x^i\ot\xd x^j+\xd x^i\ot\xd\dot x^j\right)\,.
\ee
If $g$ is symmetric (anti-symmetric), then $g^c$ is symmetric (anti-symmetric). If $g$ is non-degenerate, then $g^c$ is non-degenerate. However, if $g$ is a Riemannian metric, then $g^c$ is never Riemannian, but pseudo-Riemannian of index 0.
\begin{proposition} (\cite{Yano1967}) If $K\in\T^q_p(M)$, then
$$K^c(X_1^c,\dots,X_p^c)=\left(K(X_1,\dots,X_p)\right)^c\,,$$
for any vector fields $X_1,\dots,X_p$ on $M$.
\end{proposition}
\begin{remark}
In \cite{Grabowski1997,Grabowski1999} there is defined a lift $X\mapsto \hat X$ of sections of a Lie algebroid $\zt:E\to M$ into the space of vector fields on $E$. It is a homomorphism of Lie brackets, $\widehat{[X,Y]_E}=[\hat X,\hat Y]$. Here, the bracket on the left hand side is the Lie algebroid bracket of sections $X,Y\in\Sec(E)$, while the bracket on the right hand side is the Lie bracket of vector fields. This lift completely characterizes the Lie algebroid and can be used in a formulation of geometric mechanics on Lie algebroids \cite{Grabowska:2008,Grabowska:2011,Grabowska:2006}.
\end{remark}
\noindent There exist not only complete lifts of tensor fields, but also naturally defined complete lifts of affine connections.
\begin{theorem} (\cite{Yano1967})
Let $\nabla$ be a connection on a manifold $M$. Then there exists a unique connection $\nabla^c$ on $\sT M$ such that
$$\nabla^c_{X^c}(Y^{c})=\left(\nabla_XY\right)^{c}$$
for any vector fields $X,Y$ on $M$. Moreover, for any tensor field $K$ on $M$, we have
$$\nabla^c_{X^c}(K^c)=\left(\nabla_X\,K\right)^c\quad\text{and}\quad\nabla^c(K^c)=\left(\nabla K\right)^c\,.$$
If $\Tor^\nabla$ and $R^\nabla$ denote the torsion and the curvature of $\nabla$, then $$\Tor^{\nabla^c}=\left(\Tor^\nabla\right)^c\quad\text{and}\quad R^{\nabla^c}=\left(R^\nabla\right)^c\,.$$
\end{theorem}
\begin{corollary}
If $\nabla^g$ is the metric (Levi-Civita) connection for a pseudo-Riemannian metric $g$, then $\left(\nabla^g\right)^c$
is the metric connection for the pseudo-Riemannian metric $g^c$.
\end{corollary}
\noindent Let $\zG^i_{jk}$ be the connection components for $\nabla$ with respect to a local coordinate system $(x^i)$ on $M$. Then the coordinate components of $\nabla^c$ with respect to the induced coordinate system $(x^i,\dot x^j)$ on $\sT M$ are
\beas &\tilde\zG^i_{jk}=\zG^i_{jk}\,,\quad\tilde\zG^{\bar i}_{j\bar k}=\zG^i_{jk}\,,\quad
\tilde\zG^{\bar i}_{\bar jk}=\zG^i_{jk}\,,\quad\tilde\zG^{\bar i}_{jk}=\frac{\pa \zG^i_{jk}}{\pa x^l}\,\dot x^l\,,\\
&\tilde\zG^{i}_{j\bar k}=0\,,\quad\tilde\zG^{i}_{\bar jk}=0\,,\quad\tilde\zG^{i}_{\bar j\bar k}=0\,.
\eeas
Here the indices with bar refer to $(\dot x^i)$.
\section{Higher lifts}\label{s6}
The tangent lifts described in the previous section have been generalized to lifts from $M$ to higher tangent bundles $\sT^rM$ by Morimoto \cite{Morimoto1970}. Let us recall that $\sT^rM$ is the bundle of $r$-jets of smooth curves
$$\zg:\R\to M\,,\quad\R\ni t\mapsto\zg(t)\in M\,,$$
at $t=0$. Such jets are classes $[\zg]_r$ of curves $\zg$ which at $0\in\R$ have contact of the order $r$, i.e. have the same position, velocity, acceleration, etc. More precisely, the two curves $\gamma_1$ and $\gamma_2$ are equivalent if, for all smooth functions $f$ on $M$,
$$\frac{\xd^\zl (f\circ\zg_1)}{\xd\, t^\zl}(0)=\frac{\xd^\zl (f\circ\zg_2)}{\xd\, t^\zl}(0)\,,\ \zl=0,1,\dots,r\,.$$
Coordinates $(x^i)$ on $M$ induce coordinates $(x^i_{\zl})$ on $\sT^rM$, where $\zl=0,1,\dots,r$, as follows. When considering jets of curves, one frequently uses the notation according to which coordinates $(x^i, \dot x^j, \ddot x^k,\dots)$ of $[\gamma]_r$ are defined by derivatives with respect to $t$, namely
$$x^i(\gamma(t))=x^i+ \dot x^i\, t+\frac{1}{2!}\ddot x^i\, t^2+\frac{1}{3!}\dddot x^i\, t^3+\dots +o(t^r)\,.$$
However, we prefer another convention and coordinates $(x^i_{\zl})$ defined by
$$x^i(\gamma(t))=x^i_0+ x^i_1\, t+ x^i_2\, t^2+\cdots +x^i_r\,t^r+ o(t^r)\,.$$
This definition of adapted coordinates in $\sT^rM$ is more convenient for the lifting procedures (see e.g. Proposition \ref{hl}).

\medskip\noindent
Let us fix $r\in\N$ for the rest of this section.
\begin{definition}[\cite{Morimoto1970}] Let $f\in C^\infty(M)$, and $\zl$ be a non-negative integer not bigger than $r$. Then the $\zl$-lift of $f$ is the function $L_\zl(f)=f^{(\zl)}$ on $\sT^rM$ defined by
\be\label{fl}f^{(\zl)}([\zg]_r)=\frac{1}{\zl!}\left[\frac{\xd^{\zl}(f\circ\zg)}{\xd\,t^{\zl}}\right]_{t=0}\,,\ee
for $[\zg]_r\in\sT^rM$, where $\zg:\R\to M$ is a smooth curve representing the jet. We put by convention $f^{(\zl)}=0$ for $\zl<0$.
\end{definition}
\begin{example} For $r=2$ we have
$$\begin{array}{l}\vspace{5pt}
\displaystyle f^{(0)}\left(x_0,x_1,x_2\right)=f(x_0); \\ \vspace{5pt}
\displaystyle f^{(1)}\left(x_0,x_1,x_2\right)=\frac{\partial f}{\partial x^i}(x_0)\, x^i_1;\\
\displaystyle f^{(2)}\left(x_0,x_1,x_2\right)=\frac{1}{2}\frac{\partial^2 f}{\partial x^i\partial x^j}(x_0)\, x^i_1\, x^j_1+\frac{\partial f }{\partial x^k}(x_0)\, x^k_2\,.
\end{array}$$
\end{example}

\medskip\noindent
The lift $f\mapsto f^{(\zl)}$ is linear and the following Leibniz rule
\be\label{Lr}(f\cdot g)^{(\zl)}=\sum_{\zm=0}^\zl f^{(\zm)}\cdot g^{(\zl-\zm)}\ee is satisfied
for all $f,g\in C^\infty(M)$. It is easy to see the following.
\begin{proposition}\label{hl} For local coordinates $(x^i)$ on $M$ we have $(x^i)^{(\zl)}=x^i_\zl$.
\end{proposition}
\noindent The $\zl$-lifts of one-forms $\zw\in\zW^1(M)$ and vector fields $X\in\mathfrak{X}(M)$ are defined as follows.
\begin{theorem}\label{t6}
\
\begin{itemize}
\item
There exists a unique $\R$-linear lift $L_\zl:\zW^1(M)\to\zW^1(\sT^rM)$ such that
$$L_\zl(f\cdot \xd g)=(f\cdot \xd g)^{(\zl)}:=\sum_{\zm=0}^\zl f^{(\zm)}\xd g^{(\zl-\zm)}\,.$$
In particular, $(\xd x^i)^{(\zl)}=\xd x^i_\zl$.
\item There exists a unique $\R$-linear lift $L_\zl:\mathfrak{X}(M)\to\mathfrak{X}(\sT^rM)$ such that,
for $L_\zl(X)=X^{(\zl)}$  and for any smooth function $f$ on $M$, we have
\be\label{vf}X^{(\zl)}f^{(\zm)}=(Xf)^{(\zl+\zm-r)}\,.\ee
In particular,
$$\left(\frac{\pa}{\pa {x^i}}\right)^{(\zl)}=\frac{\pa}{\pa {x^i_{r-\zl}}}\,.$$
\end{itemize}
\end{theorem}
\noindent The lifts $f^{(r)}$, $\zw^{(r)}$, and $X^{(r)}$ will be called \emph{complete lifts to $\sT^rM$} and denoted
$f^{(c)}$, $\zw^{(c)}$, and $X^{(c)}$ if $r$ is fixed. For $r=1$, they coincide with the complete lifts from the previous section, while 0-lifts coincide with vertical lifts.
If the vector field $X\in\mathfrak{X}(M)$ induces a (local) one-parameter group of transformations $\psi_t$, then $X^{c}\in\mathfrak{X}(\sT^rM)$ induces the (local) one-parameter group of transformations $\sT^r\psi_t$ on $\sT^rM$.
\begin{example}
For $r=2$ and a vector field $X=X^i(x)\frac{\pa}{\pa x^i}$, we have in coordinates $\left(x^i_0,x^j_1,x^k_2\right)$ on $\sT^2M$,
$$\begin{array}{l}\vspace{5pt}
\displaystyle X^{(0)}=X^i(x_0) \frac{\partial}{\partial x^i_{2}}\,, \\ \vspace{5pt}
\displaystyle X^{(1)}=X^i(x_0)\frac{\partial}{\partial x^i_1}+\frac{\partial X^k}{\partial x^j}(x_0) \,x^j_1\,\frac{\partial}{\partial x^k_2}\,,\\
\displaystyle X^{(2)}=X^i(x_0)\frac{\partial}{\partial x^i}+ \frac{\partial X^k}{\partial x^j}(x_0)\, x^j_1\,\frac{\partial}{\partial x^k_1}+
        \left(\frac{1}{2}\frac{\partial^2 X^l}{\partial x^n\partial x^m}(x_0)\,x^n_1\, x^m_1+\frac{\partial X^l}{\partial x^p}(x_0)\, x^p_2\right)\frac{\partial}{\partial x^l_2}.
\end{array}$$
For a one-form $\alpha=\alpha_i(x)\,\xd x^i$, we have in turn
$$\begin{array}{l}\vspace{5pt}
\displaystyle \alpha^{(0)}=\alpha_i(x_0)\,\xd x^i; \\ \vspace{5pt}
\displaystyle \alpha^{(1)}=\frac{\partial \alpha_i}{\partial x^j}(x_0)\, x^j_1\,\xd x^i+\alpha_k(x_0)\,\xd x^k_1;\\
\displaystyle \alpha^{(2)}=\left(\frac{1}{2}\frac{\partial \alpha_i}{\partial x^k\partial x^j}(x_0)\, x^k_1\, x^j_{1}+\frac{\partial \alpha_i}{\partial x^l}(x_0)\, x^l_2\right)\,\xd x^i+
                 \frac{\partial \alpha_m}{\partial x^n}(x_0)\, x^n_1\,\xd x^m_1 + \alpha_p(x_0)\,\xd x^p_2.
\end{array}$$
\end{example}

\begin{theorem}(\cite{Morimoto1970})
If $X,Y$ are vector fields on $M$ and $f\in C^\infty(M)$, then
$$(f\cdot X)^{(\zl)}=\sum_{\zm=0}^\zl f^{(\zm)}\,X^{(\zl-\zm)}\,.$$
In particular,
\be\label{e}\left(\sum_ia_i\pa_{x^i}\right)^{(\zl)}=\sum_i\sum_{\zn=r-\zl}^ra_i^{(\zn+\zl-r)}\pa_{x^i_\zn}\,.\ee
Moreover,
\be\label{lb} [X^{(\zl)},Y^{(\zm)}]=[X,Y]^{(\zl+\zm-r)}\,.\ee
\end{theorem}

\bigskip
\noindent Finally, we apply the generalized Leibniz rule for the lifts of tensor products:
$$(T\otimes S)^{(\zl)}=\sum_{\zm=0}^\zl T^{(\zm)}\ot S^{(\zl-\zm)}\,,$$
to obtain the lifts of  $q$-contravariant tensor fields,
$$ (X_1\ot\dots\ot X_q)^{(\zl)}=\sum_{\zm_1+\dots+\zm_q=\zl}(X_1)^{(\zm_1)}\ot\dots\ot(X_q)^{(\zm_q)}\,,$$
and $p$-covariant tensor fields,
$$(\za_1\ot\dots\ot\za_p)^{(\zl)}=\sum_{\zm_1+\dots+\zm_p=\zl}(\za_1)^{(\zm_1)}\ot\dots\ot(\za_p)^{(\zm_p)}\,.$$
\noindent The lifts of symmetric tensors are symmetric, the lifts of antisymmetric tensors are antisymmetric.
Actually, we can obtain in this way the lifts of arbitrary tensor fields:
$$L_\zl:\mathscr{T}^q_p(M)\to\mathscr{T}^q_p(\sT^rM)\,,$$ so
$$L_\zl:\mathscr{T}(M)\to\mathscr{T}(\sT^rM)\,.$$  For fixed $r$, the lifts $L_r(K)=K^{(r)}$ we will call the \emph{complete lifts} of $K$ and denote with $K^{(c)}$.
\begin{proposition} (\cite{Morimoto1970}) If $K\in\mathscr{T}^q_p(M)$, then
\be\label{contraction} K^{(\zl)}(X_1^{(\zm_1)},\dots,X_p^{(\zm_p)})=
\left(K(X_1,\dots,X_p)\right)^{(\zl+\zm-r\cdot p)}\,,
\ee
where $\zm=\sum_i\zm_i$, for any vector fields $X_1,\dots,X_p$ on $M$. In particular,
$$ K^c(X_1^c,\dots,X_p^c)=\left(K(X_1,\dots,X_p)\right)^c\,.$$
\end{proposition}
\noindent Similarly to the case of the tangent bundle, there are higher lifts of affine connections.
\begin{theorem}(\cite{Morimoto1970,Yano1967})
Let $\nabla$ be a connection on a manifold $M$. There exists a unique affine connection $\nabla^c$ on $\sT^rM$ which satisfies the condition
\be\label{connection}\nabla^c_{X^{(\zl)}}K^{(\zm)}=\left(\nabla_XK\right)^{(\zl+\zm-r)}\,,\ee
for $\zl,\zm=0,\dots,r$, any vector field $X$ on $M$, and any tensor field $K$ on $M$. We have
\be\label{nablaK1}\nabla^cK^{(\zm)}=\left(\nabla K\right)^{(\zm)}\,.\ee
In particular,
$$\nabla^c_{X^c}K^c=\left(\nabla_XK\right)^c\quad\text{and}\quad \nabla^cK^c=\left(\nabla K\right)^c\,.$$
\end{theorem}
\begin{remark} Note that all these lifts are defined intrinsically and have a functorial character, since they come from the \emph{flow natural equivalence} of functors (see e.g. \cite{Kolar1996a,Kolar1996,Wamba2012}).
\be\label{fne} \zk^r:\sT^r\circ\sT\to\sT\circ\sT^r\,,\quad \za^r:\sT^r\sT^*\to\sT^*\sT^r\,,\ee
An approach to the $r$-lifts we have defined \emph{via} the above equivalences one can find in \cite{Wamba2012}.
The equivalence $\za^{1}$ was discovered by Tulczyjew \cite{Tulczyjew1974} (see also \cite[26.11]{Kolar1996} for a more categorical approach) and used to an elegant intrinsic approach to Lagrangian systems which allows for singular Lagrangians \cite{Tulczyjew1977,Tulczyjew1989}.
For $\za^r$ in full generality one can consult \cite{Cantrijn1989}. On the other hand, the first of equivalences (\ref{fne}) is a particular example of a more general equivalence, namely
$$\zk^{AB}:\sT^A\circ\sT^B\to\sT^B\circ\sT^A\,,$$
valid for all Weil functors (see, e.g. \cite[35.18]{Kolar1996}).

\medskip\noindent A nice description of $\zk^{m,n}:\sT^m\sT^n\to\sT^n\sT^m$ is the following (cf. \cite[Example 5.1]{Grabowski2012}). Let
$\sT^{m,n}M$ be the space of `homotopies' $\R^2\ni(s,t)\mapsto\zh(s,t)\in M$ modulo the equivalence
$$\zh\sim\zh'\ \Longleftrightarrow \ \frac{\pa^{k+l}\,\zh}{\pa s^{k}\,\pa t^{\,l}}(0,0)=\frac{\pa^{k+l}\,\zh'}{\pa s^k\,\pa t^{\,l}}(0,0)\,,$$
where $k=0,1,\dots,m$ and $l=0,1,\dots,l$. The class of $\zh$ we will denote $[\zh]_{m,n}$. The natural map
$$\zt^{m,n}_M:\sT^{m,n}M\to M\,\quad \zt^{m,n}([\zh]_{m,n})=\zh(0,0)$$
makes $\sT^{m,n}M$ into a fiber bundle over $M$. For a smooth map $\zf:M\to N$ we define the morphism of fiber bundles
$$\sT^{m,n}\zf:\sT^{m,n}M\to\sT^{m,n}N\,,\quad \sT^{m,n}\zf([\zh]_{m,n})=[\zf\circ\zh]_{m,n}\,,$$
which leads to the functor $\sT^{m,n}$ from the category of manifolds to the category of fiber bundles.
There are obvious equivalences of functors,
$$J^{m,n}:\sT^{m,n}\to\sT^{n,m}\,,\quad J^{m,n}([h]_{m,n})=[\bar h]_{n,m}\,,$$
where $\bar h(t,s)=h(s,t)$, and
$$I^{m,n}:\sT^{m,n}\to\sT^m\circ\sT^n\,,\quad I^{m,n}([h]_{m,n})=[h^1]_m\,,$$
where $h^1:\R\to\sT^nM$ is defined by $h^1(s)=[t\mapsto h(s,t)]_n$.
Finally,
$$\zk^{m,n}=I^{n,m}\circ J^{m,n}\circ(I^{m,n})^{-1}:\sT^m\sT^n\to\sT^n\sT^m\,.$$
\end{remark}
\section{Lifts of statistical manifolds}
Let $(M,g,T)$ be a statistical manifold. The obvious lift of this structure to $\sT^rM$ is $(\sT^rM,g^c,T^c)$. As $T$ is a symmetric $(0,3)$-tensor, the tensor $T^c$ is also a symmetric $(0,3)$-tensor, so that $(\sT^rM,g^c,T^c)$ is a statistical manifold. We will call it the $r$-\emph{lift of the statistical manifold} $(M,g,T)$. Note however, that even if $g$ is a Riemannian metric, the metric $g^c$ is only pseudo-Riemannian.
Let us  see how such a lift looks like for the other geometric objects describing a statistical manifold.
\begin{theorem}\label{statmlift} For $(M,g)$ being a pseudo-Riemannian manifold, $\nabla$ being a connection on $M$, and $\nabla^*$ being the dual connection, let $g^c$, $\nabla^c$, $\left(\nabla^*\right)^c$ be the corresponding complete lifts to the bundle $\sT^rM$.
\begin{enumerate}
\item If $g$ is a pseudo-Riemannian metric on $M$, then $g^c$ is a pseudo-Riemannian metric on $\sT^rM$.
\item If $\nabla$ is a connection on $M$ and $\Tor^\nabla$, $R^\nabla$ are the torsion and the curvature of $\nabla$, then
$$\Tor^{\nabla^c}=\left(\Tor^\nabla\right)^c\quad\text{and}\quad R^{\nabla^c}=\left(R^\nabla\right)^c\,.$$
In particular, $\nabla^c$ is torsion-free (resp., flat) if $\nabla$ is torsion-free (resp., flat).
\item The complete lift of the dual connection $\nabla^*$ is the connection dual to $\nabla^c$, $\left(\nabla^*\right)^c=\left(\nabla^c\right)^*$.
\item If $\nabla^g$ is the metric connection for $g$, then $\left(\nabla^g\right)^c$ is the metric connection for $g^c$, $\nabla^{g^c}=\left(\nabla^g\right)^c$.
\item If $\nabla$ is torsion-free and $\nabla^g=\frac{1}{2}(\nabla+\nabla^*)$, then $\nabla^c$ is torsion-free and
$$\nabla^{g^c}=\frac{1}{2}(\nabla^c+\left(\nabla^c\right)^*)\,.$$
\item If $(\nabla,\nabla^*)$ is a dualistic structure on $M$, then $(\nabla^c,\left(\nabla^c\right)^*)$ is a dualistic structure on $\sT^rM$.
\item If $\nabla$ is torsion-free and Codazzi coupled with $g$, then $\nabla^c$ is torsion-free and Codazzi coupled with $g^c$.
\end{enumerate}
\end{theorem}
\begin{remark} The above theorem expresses the functoriality of our construction.
\end{remark}
\begin{lemma} Let $M$ be a manifold of dimension $n$ and $\zp:\sT^rM\to M$ be the canonical projection. If vector fields $X_1,\dots,X_n$ generate the vector bundle $\sT M$ in a neighbourhood of $x\in M$, and $\zp(p)=x$ for some $p\in\sT^rM$, then $X_i^{(\zl)}$, $i=1,\dots,n$, $\zl=0,\dots,r$, generate the vector bundle $\sT\sT^rM$ in a neighbourhood of $p$ in $\sT^rM$.
\end{lemma}
\begin{proof}
That the vector fields $X_1,\dots,X_n$ generate $\sT M$ around $x\in M$ is equivalent to the fact that $X_1(x),\dots,X_n(x)$ is a basis of $\sT_xM$.
Hence, there is a system $(x^1,\dots,x^n)$ of local coordinates around $x$ such that $X_i(x)=\pa_{x^i}$. In a neighbourhood of $x$ we can write $X_i=f_i^j\pa_{x^j}$, $i,j=1,\dots,n$, where $f_i^i(x)=1$ and $f_i^j(x)=0$ for $i\ne j$. According to (\ref{e}),
$$X_i^{(\zl)}=\left(\sum_jf_i^j\pa_{x^j}\right)^{(\zl)}=
\sum_j\sum_{\zn=r-\zl}^r(f_i^j)^{(\zn+\zl-r)}\pa_{x^j_\zn}\,.$$
Let us first put $\zl=r$. Then we have
$$X_i^{(r)}(p)=\sum_j(f_i^j)^{(0)}(p)\pa_{x^j}+(\text{higher})
=\sum_jf_i^j(x)\pa_{x^j}+(\text{higher})=\pa_{x^i}+(\text{higher})\,.$$
The text (higher) means a vector field which is a combination of $\pa_{x^j_\zn}$ with $\zn>0$. Generally, for $\zl<r$,
$$X_i^{(\zl)}(p)=\sum_j\sum_{\zn=r-\zl}^r(f_i^j)^{(\zn+\zl-r)}(p)\pa_{x^j_\zn}
=\sum_jf_i^j(x)\pa_{x^j_{r-\zl}}
+(\text{higher})=\pa_{x^i_{r-\zl}}+(\text{higher})\,,$$
where (higher) means a vector field which is a combination of $\pa_{x^j_\zn}$ with $\zn>r-\zl$. Now it is clear that  $X_i^{(\zl)}(p)$, $i=1,\dots,n$, $\zl=0,\dots,r$, span $\sT^r_pM$.

\end{proof}
\begin{proof}[Proof of Theorem \ref{statmlift}] \

1. Using the diagonalization algorithm for symmetric matrices, we can find local linearly independent vector fields $X_1,\dots,X_n$ on $M$, $n=\dim(M)$, such that
$g(X_i(x),X_i(x))=\pm 1$. According to (\ref{contraction}),
$$g^c(X_i^{(\zl)}(p),X_i^{(r-\zl)}(p))=\left(g(X_i,X_i)\right)^{(0)}(x)=g(X_i(x),X_i(x))\pm 1\,,$$
for $i=1,\dots,n$. But $X_i^{(\zl)}$, $i=1,\dots,n$, $\zl=0,\dots,r$, locally generate $\sT\sT^rM$, so $g^c$ is non-degenerate.

\medskip
2. Using (\ref{connection}) and (\ref{contraction}), we get
\be\label{nablac}\nabla^c_{\pa_{x^i_{r-\zl}}}\!\!\!(\pa_{x^j_{r-\zm}})=\nabla^c_{(\pa_{x^i})^{(\zl)}}(\pa_{x^j})^{(\zm)}=
\left(\nabla_{\pa_{x^i}}\pa_{x^j}\right)^{(\zl+\zm-r)}\,,
\ee
so that
\beas &\Tor^{\nabla^c}\left(\pa_{x^i_{r-\zl}},\pa_{x^j_{r-\zm}}\right)=\nabla^c_{\pa_{x^i_{r-\zl}}}\!\!\!(\pa_{x^j_{r-\zm}})-
\nabla^c_{\pa_{x^j_{r-\zm}}}\!\!\!(\pa_{x^i_{r-\zl}})=\left(\nabla_{\pa_{x^i}}\pa_{x^j}-
\nabla_{\pa_{x^j}}\pa_{x^i}\right)^{(\zl+\zm-r)}\\
&=\big[\Tor^\nabla(\pa_{x^i},\pa_{x^j})\big]^{(\zl+\zm-r)}=\left(\Tor^\nabla\right)^c\left((\pa_{x^i})^{(\zl)},(\pa_{x^j})^{(\zm)}\right)
=\left(\Tor^{\nabla}\right)^c\left(\pa_{x^i_{r-\zl}},\pa_{x^j_{r-\zm}}\right)\,.
\eeas
Similarly,
\beas
& R^{\nabla^c}\left(\pa_{x^i_{r-\zl}},\pa_{x^j_{r-\zm}},\pa_{x^k_{r-\zn}}\right)=\nabla^c_{(\pa_{x^i})^{(\zl)}}\!\!
\nabla^c_{(\pa_{x^j})^{(\zm)}}(\pa_{x_k})^{(\zn)}-\nabla^c_{(\pa_{x^j})^{(\zm)}}\!\!
\nabla^c_{(\pa_{x^i})^{(\zn)}}(\pa_{x_k})^{(\zn)}\\
&=\nabla _{{{\left( {{\partial _{{x^i}}}} \right)}^{\left( \lambda  \right)}}}^c\!\!\left(\nabla_{\pa_{x^j}}\pa_{x^k}\right)^{(\zm+\zn-r)}-
\nabla _{{{\left( {{\partial _{{x^j}}}} \right)}^{\left( \zm  \right)}}}^c\!\!\left(\nabla_{\pa_{x^i}}\pa_{x^k}\right)^{(\zl+\zn-r)}\\
&=\left(\nabla_{\pa_{x^i}}\nabla_{\pa_{x^j}}\pa_{x^k}-\nabla_{\pa_{x^j}}\nabla_{\pa_{x^i}}\pa_{x^k}\right)^{(\zl+\zm+\zn-2r)}
=\left(R^\nabla(\pa_{x^i},\pa_{x^j},\pa_{x^k})\right)^{(\zl+\zm+\zn-2r)}\\
&=\left(R^\nabla\right)^c\left((\pa_{x^i})^{(\zl)},(\pa_{x^j})^{(\zm)},(\pa_{x^k})^{(\zn)}\right)=
\left(R^{\nabla}\right)^c\left(\pa_{x^i_{r-\zl}},\pa_{x^j_{r-\zm}},\pa_{x^k_{r-\zn}}\right)\,.
\eeas

3. In view of (\ref{vf}) and (\ref{nablac}), we have
\beas &g^c\left(\left(\nabla^*\right)^c_{\pa_{x^i_{r-\zl}}}\!\!\!\pa_{x^j_{r-\zm}},\pa_{x^k_{r-\zn}}\right)=
g^c\left(\left(\nabla^*_{\pa_{x^i}}\pa_{x^j}\right)^{(\zl+\zm-r)},\left(\pa_{x^k}\right)^{(\zn)}\right)\\
&=\left(g\left(\nabla^*_{\pa_{x^i}}\pa_{x^j},\left(\pa_{x^k}\right)\right)\right)^{(\zl+\zm+\zn-2r)}
\!\!\!=\big[\pa_{x^i}\big(g\left(\pa_{x^j},\pa_{x^k}\right)\big)-g\left(\pa_{x^j},\nabla_{\pa_{x^i}}\pa_{x^k}\right)\big]^{(\zl+\zm+\zn-2r)}\\
&=(\pa_{x^i})^{(\zl)}\,\big[g\left(\pa_{x^j},\pa_{x^k}\right)\big]^{(\zm+\zn-r)}-g^c\left((\pa_{x^j})^{(\zl)},
\left(\nabla_{\pa_{x^i}}\pa_{x^k}\right)^{(\zm+\zn-r)}\right)\\
&=(\pa_{x^i})^{(\zl)}\Big(g^c\left((\pa_{x^j})^{(\zm)},(\pa_{x^k})^{(\zn)}\right)\Big)-g^c\left((\pa_{x^j})^{(\zm)},
\nabla^c_{(\pa_{x^i})^{(\zl)}}(\pa_{x^k})^{(\zn)}\right)\\
&=g^c\left(\left(\nabla^c\right)^*_{\pa_{x^i_{r-\zl}}}\!\!\!\pa_{x^j_{r-\zm}},\pa_{x^k_{r-\zn}}\right)\,.
\eeas

4. Metric connections are characterized as torsion-free and self-dual. From 2) and 3) it follows that these properties
are respected when passing to the complete lifts.

5. and 6. are trivial, in view of previous points.

6. If $\nabla g$ is symmetric, then according to (\ref{nablaK1}),
$$\nabla^c g^c=\left(\nabla g\right)^c$$
is symmetric.
\end{proof}
\noindent Note that the complete $r$-lifts of a statistical manifold are independent on the choice of description: in terms of the $(0,3)$ tensor, or in terms of the connection and its dual. Actually, $\left(\nabla^\za\right)^c=(\nabla^c)^\za$. Indeed, if
$$g\left(\nabla^\za_XY,Z \right)=g\left(\nabla^g_XY,Z \right)-\frac{\za}{2}T(X,Y,Z)\,,$$
then
\beas &g^c\left(\left(\nabla^\za\right)^c_{X^{(\zl)}}Y^{(\zm)},Z^{(\zn)} \right)=
g^c\left(\left(\nabla_XY\right)^{(\zl+\zm-r)},Z^{(\zn)}\right)=\big[g\left(\nabla_XY,Z\right)\big]^{(\zl+\zm+\zn-2r)}\\
&=\big[g\left(\nabla^g_XY,Z\right)-\frac{\za}{2}T(X,Y,Z)\big]^{(\zl+\zm+\zn-2r)}=
g^c\left(\left(\nabla^g\right)^c_{X^{(\zl)}}Y^{(\zm)},Z^{(\zn)}\right)-T^c\left(X^{(\zl)},Y^{(\zm)},Z^{(\zn)}\right)\\
&=g^c\left((\nabla^c)^g_{X^{(\zl)}}Y^{(\zm)},Z^{(\zn)}\right)-T^c\left(X^{(\zl)},Y^{(\zm)},Z^{(\zn)}\right)
=g^c\left((\nabla^c)^\za_{X^{(\zl)}}Y^{(\zm)},Z^{(\zn)} \right)\,.
\eeas
\begin{example}\label{tlFR}
The tangent lift of the Fisher-Rao metric is (cf. (\ref{g-lift}))
$$\left(\int_\zW \bigg[\frac{\pa^2\,\log p(\zx,x)}{\pa x^k\,\pa x^i}\,\frac{\pa\,\log p(\zx,x)}{\pa x^j}+\frac{\pa\,\log p(\zx,x)}{\pa x^i}\,\frac{\pa^2\,\log p(\zx,x)}{\pa\,x^k\pa\, x^j}\bigg]\,p(\zx,x)\,\xd\zx\right)\dot x^k\,\xd x^i\ot\xd x^j$$
$$+\left(\int_\zW \frac{\pa\,\log p(\zx,x)}{\pa\, x^i}\,\frac{\pa\,\log p(\zx,x)}{\pa\, x^j}\,\frac{\pa\,\log p(\zx,x)}{\pa\, x^k}\,p(\zx,x)\,\xd\zx\right)\dot x^k\,\xd x^i\ot\xd x^j
$$
$$+\left(\int_\zW \frac{\pa\,\log p(\zx,x)}{\pa\, x^i}\,\frac{\pa\,\log p(\zx,x)}{\pa\, x^j}
\,p(\zx,x)\,\xd\zx\right)\left(\xd\dot x^i\ot\xd x^j+\xd x^i\ot\xd\dot x^j\right)\,.
$$
\end{example}
\begin{example}
The Gaussian manifold $M$ is the manifold of probability densities
$$f(\zx,\mu,\sigma)=\frac{1}{\sqrt{2\pi}\sigma}\exp\left(-\frac{(\zx-\mu)^2}{2\sigma^2}\right)$$
on the real line. It is two-dimensional with global coordinates $(\mu,\sigma)$, where $\sigma>0$, and equipped with the metric
$$g=\frac{1}{\sigma^2}(d\mu\otimes d\mu+2d\sigma\otimes d\sigma)$$
and the skewness tensor
$$T=\frac{2}{\sigma^3}(d\mu\otimes d\mu\otimes d\sigma+d\mu\otimes d\sigma\otimes d\mu+d\sigma\otimes d\mu\otimes d\mu)+
\frac{8}{\sigma^3}d\sigma\otimes d\sigma\otimes d\sigma.$$
The Christoffel symbols of the Levi-Civita connection associated with $g$ read
$$\Gamma^\mu_{\mu\mu}=\Gamma^\mu_{\sigma\sigma}=\Gamma^\sigma_{\mu\sigma}=\Gamma^\sigma_{\sigma\mu}=0\,, \quad
\Gamma^\mu_{\mu\sigma}=\Gamma^\mu_{\sigma\mu}=\Gamma^\sigma_{\sigma\sigma}=-\frac{1}{\sigma}\,, \quad
\Gamma^\sigma_{\mu\mu}=\frac{1}{2\sigma}\,.$$
In a more compact way, we can describe the above connection by specifying vector fields $\mathcal{M}$ and $\mathcal{S}$ on $\sT M$ being horizontal lifts of $\partial_\mu$ and $\partial_\sigma$, respectively. In coordinates
$(\mu,\sigma,\dot\mu,\dot\sigma)$ they read
$$\mathcal{M}=\partial_\mu+\frac{\dot\sigma}{\sigma}\partial_{\dot\mu}-\frac{\dot\mu}{2\sigma}\partial_{\dot\sigma}, \qquad
\mathcal{S}=\partial_\sigma+\frac{\dot\mu}{\sigma}\partial_{\dot\mu}+\frac{\dot\sigma}{\sigma}\partial_{\dot\sigma}.$$
The $\alpha$-connection described by connection coefficients reads
$$\aGamma{}^\mu_{\mu\mu}=\aGamma{}^\mu_{\sigma\sigma}=\aGamma{}^\sigma_{\mu\sigma}=\aGamma{}^\sigma_{\sigma\mu}=0\,, \quad
\aGamma{}^\mu_{\mu\sigma}=\aGamma{}^\mu_{\sigma\mu}=-\frac{\alpha+1}{\sigma}\,, \quad
\aGamma{}^\sigma_{\sigma\sigma}=-\frac{2\alpha+1}{\sigma}\,,\quad
\aGamma{}^\sigma_{\mu\mu}=\frac{1-\alpha}{2\sigma}\,,$$
while the horizontal vector fields are
$$\overset{\alpha}{\mathcal{M}}=\partial_\mu+\frac{(1+\alpha)\dot\sigma}{\sigma}\partial_{\dot\mu}
-\frac{(1-\alpha)\dot\mu}{2\sigma}\partial_{\dot\sigma}\,, \qquad
\overset{\alpha}{\mathcal{S}}=\partial_\sigma+\frac{(1+\alpha)\dot\mu}{\sigma}\partial_{\dot\mu}
+\frac{(1+2\alpha)\dot\sigma}{\sigma}\partial_{\dot\sigma}\,.$$

\noindent The commutator $[\overset{\alpha}{\mathcal{M}},\overset{\alpha}{\mathcal{S}}]$ is
$$[\overset{\alpha}{\mathcal{M}},\overset{\alpha}{\mathcal{S}}]=\frac{(1-\alpha^2)}{2\sigma^2}
\left(2\dot\sigma\partial_{\dot\mu}-\dot\mu\partial_{\dot\sigma}\right)\,.$$
From the above formulae we can see that $\alpha$-connections are flat for $\alpha=\pm 1$. One can therefore find two global coordinate systems on $M$ such that the Christoffel symbols of the $1$-connection and of the $(-1)$-connection, respectively, are all zero. For the $1$-connection the appropriate coordinate system is
$$x^1=\frac{\mu}{\sigma^2}, \qquad x^2=-\frac{1}{2\sigma^2}\,.$$
The above coordinate system is referred to in the literature as the system of {\it natural coordinates} on Gaussian manifold. For the $(-1)$-connection the appropriate coordinate system is given by the first two momenta of the Gaussian distribution, namely
$$m^1=\mu, \qquad m^2=\mu^2+\sigma^2\,.$$
\noindent The tangent lift $g^c$ is a pseudo-Riemannian metrics on $\sT M$ described in coordinates by (\ref{g-lift}). It can be defined geometrically as follows. The metric tensor $g$ gives rise to the vector bundle isomorphism
$$G:\sT M\longrightarrow \sT^\ast M, \qquad G(v)=g(v,\cdot).$$
The corresponding isomorphism $G^c$ for $g^c$ is a map $G^c:\sT\sT M\rightarrow\sT^\ast\sT M$ defined by
$$G^c=\alpha_M\circ\sT G\circ \kappa_M,$$
where $\kappa_M: \sT\sT M\rightarrow \sT\sT M$ is the canonical flip on $\sT\sT M$, and $\alpha_M: \sT\sT^\ast M\rightarrow \sT^\ast \sT M$ is its dual (cf. (\ref{fne})) called the \emph{Tulczyjew isomorphism} \cite{Tulczyjew1977,Tulczyjew1989}.
Translating $G^c$ back to $g^c$, we get in coordinates $(\mu,\sigma,\dot\mu,\dot\sigma)$,
$$g^c=\frac{1}{\sigma^3}[-2\dot\sigma d\mu\otimes d\mu-4\dot\sigma d\sigma\otimes d\sigma+
\sigma(d\dot\mu\otimes d\mu+d\mu\otimes d\dot\mu)+2\sigma(d\dot\sigma\otimes d\sigma+d\sigma\otimes d\dot\sigma)
]\,;$$
in matrix form
$$\frac{1}{\sigma^2}\left[\begin{array}{cccc}
                      -2\dot\sigma & 0 & \sigma & 0 \\
                      0 & -4\dot\sigma & 0 & 2\sigma \\
                      \sigma & 0 & 0 & 0 \\
                      0 & 2\sigma & 0 & 0
                    \end{array}\right]\,.$$
It is easy to see that $g^c$ has signature $(2,2)$, so it is of index 0 as stated in section \ref{sec:4}.

\noindent The components of the lifted skewness tensor $T^c$ are either $0$ or equal (by symmetric permutation of indices)
$$(T^c)_{\mu\mu\sigma}=-\frac{6\dot\sigma}{\sigma^4}\,, \qquad
(T^c)_{\sigma\sigma\sigma}=-\frac{24\dot\sigma}{\sigma^4}\,,$$
$$(T^c)_{\dot\mu\mu\sigma}=(T^c)_{\mu\mu\dot\sigma}=\frac{2}{\sigma^3}\,,\qquad (T^c)_{\dot\sigma\sigma\sigma}=\frac{8}{\sigma^3}\,.$$
\noindent The Levi-Civita and $\alpha$-connections can be lifted to $\sT M$. Since the dimension of the manifold grows, so is the number of the Christoffel symbols. It is then more convenient to look at the horizontal distribution. The horizontal distribution of the lifted Levi-Civita connection is the distribution on $\sT\sT M$ obtained as the image by $\sT\kappa_M$ of the subbundle of $\sT(\sT\sT M)$ which is spanned by the vertical and the complete lifts of $\mathcal{M}$ and $\mathcal{S}$.
Using coordinates $(\mu,\sigma,\dot\mu,\dot\sigma)$ in $\sT M$ and
$(\mu,\sigma,\dot\mu,\dot\sigma, \delta\mu,\delta\sigma,\delta\dot\mu,\delta\dot\sigma )$ in $\sT\sT M$, we can write the horizontal lifts of
$\partial_\mu$, $\partial_{\dot\mu}$, $\partial_\sigma$, $\partial_{\dot\sigma}$ with respect to the lifted Levi-Civita connection as
$$\begin{array}{l}\vspace{10pt}
\displaystyle \mathcal{X}_\mu=\partial_\mu-\frac{\delta\mu}{2\sigma}\partial_{\delta\sigma}
+\frac{\delta\sigma}{\sigma}\partial_{\delta\mu}+
\left(\frac{\dot\sigma\delta\mu}{2\sigma^2}-\frac{\delta\dot\mu}{2\sigma}\right)\partial_{\delta\dot\sigma}+
\left(-\frac{\dot\sigma\delta\sigma}{\sigma^2}+\frac{\delta\dot\sigma}{2\sigma}\right)\partial_{\delta\dot\mu}\,, \\ \vspace{10pt}
\displaystyle \mathcal{X}_{\dot\mu}=\partial_{\dot\mu}-\frac{\delta\mu}{2\sigma}\partial_{\delta\dot\sigma}+
\frac{\delta\sigma}{\sigma}\partial_{\delta\dot\mu}\,, \\ \vspace{10pt}
\displaystyle \mathcal{X}_\sigma=\partial_{\sigma}+\frac{\delta\mu}{\sigma}\partial_{\delta\mu}
+\frac{\delta\sigma}{\sigma}\partial_{\delta\sigma}+
\left(\frac{-\delta\mu\dot\sigma}{\sigma^2}+\frac{\delta\dot\mu}{\sigma}\right)\partial_{\delta\dot\mu}+
\left(\frac{-\delta\sigma\dot\sigma}{\sigma^2}+\frac{\delta\dot\sigma}{\sigma}\right)\partial_{\delta\dot\sigma}\,, \\ \vspace{10pt}
\displaystyle \mathcal{X}_{\dot\sigma}=\partial_{\dot\sigma}+\frac{\delta\mu}{\sigma}\partial_{\delta\dot\mu}
+\frac{\delta\sigma}{\sigma}\partial_{\delta\dot\sigma}\,.
\end{array}$$
From the above formulae we obtain the Christoffel symbols of the lifted Levi-Civita connection (symbols not listed below vanish):
$$\Gamma^\mu_{\mu\sigma}=\Gamma^\mu_{\sigma\mu}=\Gamma^\sigma_{\sigma\sigma}=-\frac{1}{\sigma}\,, \quad
\Gamma^\sigma_{\mu\mu}=\frac{1}{2\sigma}\,.$$
$$\Gamma^{\dot\mu}_{\mu\sigma}=\Gamma^{\dot\mu}_{\sigma\mu}=\frac{\dot\sigma}{\sigma^2}\,, \quad
\Gamma^{\dot\mu}_{\dot\mu\sigma}=\Gamma^{\dot\mu}_{\sigma\dot\mu}=-\frac{1}{\sigma}\,, \quad
\Gamma^{\dot\mu}_{\mu\dot\sigma}=\Gamma^{\dot\mu}_{\dot\sigma\mu}=-\frac{1}{\sigma}\,,$$
$$\Gamma^\sigma_{\mu\mu}=\frac{1}{2\sigma}, \quad \Gamma^\sigma_{\sigma\sigma}=-\frac{1}{\sigma}\,,$$
$$\Gamma^{\dot\sigma}_{\sigma\sigma}=\frac{\dot\sigma}{\sigma^2}\,, \qquad
\Gamma^{\dot\sigma}_{\mu\mu}=-\frac{\dot\sigma}{2\sigma^2}\,, \qquad
\Gamma^{\dot\sigma}_{\mu\dot\mu}=\Gamma^{\dot\sigma}_{\dot\mu\mu}=\frac{1}{2\sigma}\,,\quad
\Gamma^{\dot\sigma}_{\sigma\dot\sigma}=\Gamma^{\dot\sigma}_{\dot\sigma\sigma}=-\frac{1}{\sigma}\,.$$
If we are patient enough, we can do the same for the $\alpha$-connection. We can also follow the rules of calculating the Christoffel symbols for the lifted connection which are the following: all Christoffel symbols with three dotted coordinates are zero; the symbols with one dotted up and one dotted down are equal to these with no dots, while the symbols with two dotted down and no dots up are zero; the symbols with one dotted coordinates are non-zero only when the dotted coordinate is up and then they are tangent lifts of non-dotted symbols; non-dotted symbols are the same as for the original connection on $M$. The table for the Christoffel symbols for the lifted $\alpha$-connection is then the following.
$$\aGamma{}^{\mu}_{\mu\sigma}=\aGamma{}^{\mu}_{\sigma\mu}=\aGamma{}^{\dot\mu}_{\dot\mu\sigma}
=\aGamma{}^{\dot\mu}_{\dot\sigma\mu}=
\aGamma{}^{\dot\mu}_{\mu\dot\sigma}=\aGamma{}^{\dot\mu}_{\sigma\dot\mu}=-\frac{\alpha+1}{\sigma}\,,$$
$$\aGamma{}^{\sigma}_{\sigma\sigma}=\aGamma{}^{\dot\sigma}_{\dot\sigma\dot\sigma}=\aGamma{}^{\sigma}_{\sigma\sigma}
=-\frac{2\alpha+1}{\sigma}\,, \qquad
\aGamma{}^{\sigma}_{\mu\mu}=\aGamma{}^{\dot\sigma}_{\dot\mu\mu}=\aGamma{}^{\dot\sigma}_{\mu\dot\mu}
=\frac{1-\alpha}{2\sigma}\,,$$
$$\aGamma{}^{\dot\sigma}_{\sigma\sigma}=\frac{(2\alpha+1)\dot\sigma}{\sigma^2}\,,\qquad
\aGamma{}^{\dot\sigma}_{\mu\mu}=-\frac{(1-\alpha)\dot\sigma}{2\sigma^2}\,,\qquad
\aGamma{}^{\dot\mu}_{\mu\sigma}=\aGamma{}^{\dot\mu}_{\sigma\mu}=\frac{(\alpha+1)\dot\sigma}{\sigma^2}\,.$$
\end{example}
\begin{example}
Let $x=(x^1,\ldots, x^n)\in\R^n=M$ and $\zx=(\zx^1,\ldots, \zx^n)\in\R^n$. The exponential family is a family of probability measures of the form
$$p(\zx,x)=\exp\left(k(\zx)+\lan x\,|\,\zx\ran-\psi(x)\right)d^n\zx\,,$$
where $k(\zx)$ is a normalizing factor such that $\int p(\zx,x)d^n\zx=1$ and $\psi$ is some convex function on $\R^n$. The Fisher information matrix is calculated according to the usual formula,
$$g_{ij}(x)=\int \left(\frac{\partial\log p(\zx,x)}{\pa x^i}\,\frac{\partial\log p(\zx,x)}{\pa x^j}\right)\,p(\zx,x)\,d^n\zx=\frac{\partial^2\psi}{\pa x^j\pa x^i}(x)\,.$$
We have used here some regularity assumptions, namely that we can differentiate under the integral sign. We have then
$$0=\frac{\partial}{\pa x^i}\left(\int p(\zx,x)d^n\zx\right)=\int\frac{\partial \log p(
\zx,x)}{\pa x^i}\,p(\zx,x)\,d^n\zx\,.$$
Differentiating once more, we get
\begin{multline*}0=\frac{\partial}{\pa x^j}\left(\int\frac{\partial \log p(
\zx,x)}{\pa x^i}\,p(\zx,x)\,d^n\zx\right)=\\
\int\left(\frac{\partial^2 \log p(\zx,x)}{\pa x^j\pa x^i}+\frac{\partial\log p(\zx,x)}{\pa x^i}\,\frac{\partial\log p(\zx,x)}{\pa x^j}\right)p(\zx,x)\,d^n\zx=\\
-\frac{\partial^2 \psi}{\pa x^j\pa x^i}(x)+\int\left(\frac{\partial\log p(\zx,x)}{\pa x^i}\,\frac{\partial\log p(\zx,x)}{\pa x^j}\right)p(\zx,x)\,d^n\zx\,.
\end{multline*}
The Riemannian metric for the statistical model of the exponential family in natural coordinates $(x^i)$ is given then by the Hessian of the function $\psi$,
$$g(x)=\frac{\partial^2 \psi}{\pa x^j\pa x^i}(x)\, \xd x^i\otimes \xd x^j\,.$$
The complete lift of $g$ to $\sT M$ reads then
$$g^c(x,\dot x)=\frac{\partial^3 \psi}{\pa x^k\pa x^j\pa x^i}(x)\,\dot x^k\, \xd x^i\otimes \xd x^j+\frac{\partial^2 \psi}{\pa x^j\pa x^i}(x)\left(\xd \dot x^i\otimes\xd x^j+\xd x^i\otimes \xd\dot x^j\right)\,.$$
It is given by the Hessian of the complete lift $\psi^c$ of the function $\psi$,
$$\psi^{c}(x,\dot x)=\frac{\partial\psi}{\pa x^i} (x)\,\dot x^i\,.$$
\noindent The Levi-Civita connection for the exponential family is conveniently described by the Christoffel symbols $\Gamma_{ijk}=g_{kl}\,\Gamma^{l}_{ij}$. We have
$$\Gamma_{ijk}=\frac{1}{2}\frac{\partial^3 \psi}{\pa x^k\pa x^j\pa x^i}\,.$$
Let $\bar\Gamma^*_{**}$ denote the Christoffel symbols of the Levi-Civita conection lifted to $\sT M$. Using general rules, we get
$$\bar\Gamma^{\bar i}_{\bar j\bar k}=\bar\Gamma^{i}_{\bar j\bar k}=\bar\Gamma^{i}_{j \bar k}=\bar\Gamma^{\bar i}_{\bar j k}=0\,,$$
$$\bar\Gamma^{\bar i}_{\bar j k}=\bar\Gamma^{\bar i}_{j\bar k}=\Gamma^{i}_{jk}\,, \qquad
\bar\Gamma^{\bar i}_{jk}=(\Gamma^{i}_{jk})^{(1)}\,.$$
Taking into account the components of the lifted metric, we can calculate the Christoffel symbols of the lifted connection with lowered index. What we get is
$$\bar\Gamma_{ijk}=(\Gamma_{ijk})^{(1)}, \qquad\bar \Gamma_{\bar ijk}=\bar \Gamma_{i\bar jk}=\bar \Gamma_{ij\bar k}=\Gamma_{ijk}\,,$$
while coefficients with two or three `lifted' indices vanish. For the exponential family it means precisely that the coefficients of the lifted connection with lowered index are third order partial derivatives of the complete lift of $\psi$.

\noindent The skewness tensor $T$ is also given by a very simple formula:
$$T_{ijk}=\frac{\partial^3 \psi}{\pa x^k\pa x^j\pa x^i}\,.$$
It is easy to see now that the connection $\nabla$ is in this case flat, and the natural coordinate system $(x^i)$ is an affine system for this connection.
The tangent lift $T^c$ of $T$ is again given by the complete lift of the function $\psi$,
namely
$$(T^c)_{ijk}=\frac{\partial^3 \psi^c}{\pa x^k\pa x^j\pa x^i}
\quad\text{and}\quad (T^c)_{\bar{i}jk}=\frac{\partial^3 \psi^c}{\pa\bar x^i\pa x^j\pa x^k}
=\frac{\partial^3 \psi}{\pa x^i\pa x^j\pa x^k}\,.$$
Since $\psi^{c}$ is linear in the $(\dot x^i)$ coordinates, higher derivatives of $\psi^{c}$ with respect of these coordinates are $0$, so the corresponding coefficients of $T^c$ vanish as well. Summarising, the geometry of exponential family in natural coordinates is defined by the function $\psi$: the metric is just the Hessian of $\psi$, while the Levi-Civita connection  and the skewness tensor are expressed by the third-order partial derivatives of $\psi$. The lifted geometry is associated with the complete lift of $\psi$. In the adapted coordinates, the metric, the metric connection, and the skewness tensor are given by the same formulae, but for the complete lift of $\psi$.

The above formulae for the lifted metric, the lifted connection, and the lifted skewness tensor were calculated `by hand'. We can, however, consider also higher lifts of the statistical model given by an exponential family. For any natural number $r$, the lift of the metric tensor to $\sT^r M$ reads  in coordinates
$$\left(g_{kl}(x)\,\xd x^k\otimes \xd x^l\right))^{c}=\sum_{\mu,\nu\geq 0}g_{kl}^{(r-\mu-\nu)}(x)\,\xd x^{k}_\mu\otimes \xd x^{l}_\nu\,,$$
which means that
$$g^{(r)}(\partial_{x^k_\mu}, \partial_{x^l_\nu})=g_{kl}^{(r-\mu-\nu)}\,.$$
We use the convention according to which $f^{(p)}=0$ if $p<0$.
For $g$ given in the privileged coordinates by the Hessian of a function $\psi$, we get
$$g^{(r)}(\partial_{x^k_\mu},\partial_{x^l_\nu})=\left(\frac{\partial^2\psi}{\pa x^k\partial x^l}\right)^{\!(r-\mu-\nu)}=\frac{\partial^2\psi^c}{\pa x^k_\mu\partial x^l_\nu}\,,$$
which means that the lifted metric is given by the Hessian of the lifted function. The same is valid for the Levi-Civita connection and the skewness tensor -- both are defined with the use of $\psi^{c}$ by the formulae
$$\Gamma^{c}_{(i,\mu)(j,\nu)(k,\lambda)}=\frac{1}{2}\frac{\partial^3 \psi^c}{\pa x^i_\mu\,\pa x^j_\nu\,\pa x^k_\zl}\,,
\qquad T^{c}_{(i,\mu)(j,\nu)(k,\lambda)}=\frac{\partial^3 \psi^c}{\pa x^i_\mu\,\pa x^j_\nu\,\pa x^k_\zl}\,.$$
\end{example}
\noindent Our next observation is that we can lift also statistical manifolds admitting torsion (SMATs).
\begin{theorem}
 If $(M,g,\nabla)$ is SMAT then $(\sT^rM,g^c,\nabla^c)$ is a SMAT.
\end{theorem}
\begin{proof} Let $X,Y,Z$ be arbitrary vector fields on $M$. Since If $(M,g,\nabla)$ is a SMAT, we have
$$(\nabla_Xg)(Y,Z)-(\nabla_Yg)(X,Z)=-g\left(\Tor^\nabla(X,Y),Z\right)\,.$$
Hence (cf. (\ref{contraction}) and (\ref{connection})),
\beas &(\nabla^c_{X^{(\zl)}}\,g^c)\left(Y^{(\zm)},Z^{(\zn)}\right)-(\nabla^c_{Y^{(\zm)}}\,g^c)\left(X^{(\zl)},Z^{(\zn)}\right)\\
&=(\nabla_X\,g)^{(\zl)}\left(Y^{(\zm)},Z^{(\zn)}\right)-(\nabla_Y\,g)^{(\zm)}\left(X^{(\zl)},Z^{(\zn)}\right)\\
&=\big[(\nabla_X\,g)(Y,Z)-(\nabla_Y\,g)(X,Z)\big]^{(\zl+\zm+\zn-2r)}=-\big[g\left(\Tor^\nabla(X,Y),Z\right)\big]^{(\zl+\zm+\zn-2r)}\\
&=-g^c\left((\Tor^\nabla(X,Y))^{(\zl+\zm-r)},Z^{(\zn)}\right)=-g^c\left((\Tor^\nabla(X,Y))^{(\zl+\zm-r)},Z^{(\zn)}\right)\\
&=-g^c\left(\left(\Tor^\nabla\right)^c\left(X^{(\zl)},Y^{(\zm)}\right),Z^{(\zn)}\right)=
-g^c\left((\Tor^{\nabla^c})\left(X^{(\zl)},Y^{(\zm)}\right),Z^{(\zn)}\right)\,,
\eeas
so that  the pair $(g^c,\nabla^c)$ is torsion coupled.
\end{proof}
\begin{remark} In \cite{Peyghan2021} the authors also consider various lifts of tensor fields and connections associated with statistical manifold structures to the tangent bundle, but no explicit concept of the lifted statistical manifold is presented. The lifted metrics are variants of the induced \emph{Sasaki metric} on $\sT M$. The Sasaki metric is a canonical Riemannian metric on the tangent bundle of a Riemannian manifold equipped with an affine connection. It was originally discovered by Sasaki  \cite{Sasaki1958} and expanded on by Dombrowski \cite{Dombrowski1962}. An advantage of this approach is that the Sasaki metric is again Riemannian, but a disadvantage is that it seems not to be functorially related to any canonical lift of the corresponding contrast function. Also, there are infinitely many Sasaki-like Riemannian metrics on $\sT M$ associated with a Riemannian manifold $M$ (some of them were considered in \cite{Peyghan2021}), so any choice is arbitrary and therefore not functorial.
\end{remark}

\section{Lifts of contrast functions}
The main observation of this section is the following.
\begin{theorem}\label{main}
If $F:M\ti M\to\R$ is a contrast function on $M$, then $$F^c:\sT^r(M\ti M)=\sT^rM\ti\sT^rM\to\R$$ is a contrast function on $\sT^rM$. Moreover, the metric $g^{F^c}$, the connection $\nabla^{F^c}$, its dual $\left(\nabla^{F^c}\right)^*$, and the $(0,3)$-tensor $T^{F^c}$ induced by $F^c$ are the $r$-lifts of $g^F$, $\nabla^F$, $\left(\nabla^F\right)^*$, and $T^F$, respectively.
\end{theorem}
\begin{proof}
Let $(x^i)$ be local coordinates on $M$ in a neighbourhood of $p\in M$, $x^i(p)=0$. Let $(y^i)$ be the same coordinates on another copy of $M$, so that $(x^i,y^j)$ are local coordinates in a neighbourhood of $(p,p)\in M\ti M$. Let $(\bar x,\bar y)$ be the adapted coordinates $\bar x^i=(x^i_0,\dots,x^i_r)$ and $\bar y^j=(y^j_0,\dots,y^j_r)$ on $\sT^rM\ti\sT^rM$. According to Theorem \ref{thcontrast}, as $F$ is a contrast function in a neighbourhood of $(p,p)$, we can write $F$ in the form
$$
F(x,y)=\frac{1}{2}\,(x^i-y^i)(x^j-y^j)\,h_{ij}(x,y)\,,
$$
where $h_{ij}=h_{ji}$ and $[h_{ij}(0,0)]$ is an invertible matrix, i.e. $[h_{ij}(x,y)]$ is invertible in a neighbourhood of $(0,0)$. We have then (cf. (\ref{Lr}) and Proposition \ref{hl})
$$F^c(\bar x,\bar y)=\frac{1}{2}\sum_{\zl+\zm+\zn=r}\,(x^i_\zl-y^i_\zl)(x^j_\zm-y^j_\zm)\,(h_{ij})^{(\zn)}(\bar x,\bar y)\,.$$
As in the proof of Theorem \ref{thcontrast}, this implies that the first jets of $F^c$ vanish on the diagonal $\zD_{\sT^rM}$. To finish the proof, it suffices to show that the matrix $[h_{(i,\zl)(j,\zm)}(\bar x,\bar x)]$ is invertible for points $\bar x$ whose projections on $M$ are sufficiently close to $p$, where $h_{(i,\zl)(j,\zm)}=(h_{ij})^{(r-\zl-\zm)}$. This is equivalent to the fact that the matrix $[h_{(i,\zl)(j,\zm)}(\bar x,\bar y)]$ is invertible for $(\bar x,\bar y)$ in a neighbourhood of $(\bar 0,\bar 0)$. This, in turn, is equivalent to the fact that the tensor field
$$H(\bar x)=h_{(i,\zl)(j,\zm)}(\bar x,\bar x)\,\xd x^i_\zl\ot \xd x^j_\zm$$
is a pseudo-Riemannian metric in a neighbourhood of $p\in M$. But
$$H(\bar x)=\left(h_{ij}^{(r-\zl-\zm)}\,\xd x^i_\zl\ot \xd x^j_\zm\right)(\bar x,\bar x)=\left(h_{ij}\,\xd x^i\ot \xd x^j\right)^c(\bar x,\bar x)\,.$$
Since $F$ is a contrast function, the tensor $h(x)=h_{ij}(x,x)\,\xd x^i\ot \xd x^j$ is the pseudo-Riemannian metric $g^F$, so $H$ is its complete lift, that proves that $F^c$ is a contrast function and $g^{F^c}=\left(g^F\right)^c$.
It is obvious that $\left(F^c\right)^*=\left(F^*\right)^c$. The connection $\nabla^{F^c}$ and the skewness tensor $T^{F^c}$ are defined as
\beas
& g^{F^c}\left(\nabla^{F^c}_{X^{(\zl)}}\,Y^{(\zm)},Z^{(\zn)}\right)(\bar x)=-F^c[X^{(\zl)}\,Y^{(\zm)}|\,Z^{(\zn)}](\bar x,\bar x)\,,\\
& T^{F^c}(X^{(\zl)},Y^{(\zm)},Z^{(\zn)})(\bar x)=(F^c-\left(F^c\right)^*)[X^{(\zl)}|\,Y^{(\zm)}\,Z^{(\zn)}](\bar x,\bar x)\,.
\eeas
It is easy to see that $^i(X^{(\zl)})=(^i\!X)^{(\zl)}$ for $i=1,2$. We have then (cf. Theorem \ref{t6})
\beas &-F^c[X^{(\zl)}\,Y^{(\zm)}|\,Z^{(\zn)}](\bar x,\bar x)=(^1\!X)^{(\zl)}\,(^1Y)^{(\zm)}\,(^2Z)^{(\zn)}\,(F^c)(\bar x,\bar x)\\
&=(^1\!X\,^1Y\,^2Z\,(F))^{(\zl+\zm+\zn-2r)}(\bar x,\bar x)=\left(g^F(\nabla_X\,Y,Z)\right)^{(\zl+\zm+\zn-2r)}(\bar x)\\
&=\left(g^F\right)^c\left((\nabla^F_X\,Y)^{(\zl+\zm-r)},Z^{(\zn)}\right)(\bar x)=g^{F^c}\left(\left(\nabla^{F}\right)^c_{X^{(\zl)}}\,Y^{(\zm)},Z^{(\zn)}\right)(\bar x)\,.
\eeas
This proves $\left(\nabla^{F}\right)^c=\nabla^{F^c}$. Similarly one can prove $\left(\nabla^{F^*}\right)^c=\left(\nabla^{F^c}\right)^*$.
For the skewness tensor we have
\beas &\left(F^c-\left(F^c\right)^*\right)[X^{(\zl)}|\,Y^{(\zm)}\,Z^{(\zn)}](\bar x,\bar x)=(^1\!X)^{(\zl)}\,(^2Y)^{(\zm)}\,(^2Z)^{(\zn)}\,\left(F^c-\left(F^c\right)^*\right)(\bar x,\bar x)\\
&=(^1\!X\,^2Y\,^2Z\,(F-F^*))^{(\zl+\zm+\zn-2r)}(\bar x,\bar x)=[T^{F}(X,Y,Z)]^{(\zl+\zm+\zn-2r)}(\bar x)\\
&=\left(T^F\right)^c\left(X^{(\zl)},Y^{(\zm)},Z^{(\zn)}\right)\,,
\eeas
that shows $T^{F^c}=\left(T^F\right)^c$.
\end{proof}
An obvious adaptation of the above proof, this time with the use of Theorem \ref{thcontrast1}, gives immediately the following.
\begin{theorem}
If $F:N\to\R$ is a contrast function on a closed submanifold $N_0\subset N$, then $F^c:\sT^rN\to\R$ is a contrast function on the closed submanifold $\sT^rN_0\subset\sT^rN$.
\end{theorem}
\begin{example}
The tangent lift of the Kullback-Leibler divergence (\ref{KL}) reads
$$D_{KL}^{(1)}(x,y,\dot x,\dot y)=\int_\zW\,\left(\frac{\pa\,\log p(\zx,x)}{\pa x^k}\bigg[\log\frac{p(\zx,x)}{p(\zx,y)}+1\bigg]\cdot\dot x^k-\frac{\pa\,\log p(\zx,y)}{\pa y^k}\cdot\dot y^k\right)\,p(\zx,x)\,\xd\,\zx\,.
$$
\end{example}
\noindent It is a matter of direct calculations to see that the pseudo-Riemannian metric $g^{D_{KL}^{(1)}}$ is the tangent lift of the Fisher-Rao metric described in Example \ref{tlFR}.

\begin{remark} Already after publishing the first version of this paper in \emph{arXiv}, our attention was turned to the paper \cite{Matsuzoe2003}.  The authors define there lifts of statistical manifold structures analogous to ours, but only to tangent bundles. Using horizontal lifts, they obtain also lifted statistical manifolds, this time with Riemannian metrics. The metrics are actually Sasaki metrics. The problem is that the Sasaki metrics (also those associated with an extra connection) are not functorially associated with lifts of the corresponding contrast functions, thus is not clear how to obtain contrast functions for the structures lifted in this way. Moreover, since in the presence of a connection, the tangent space $\sT_{v_x}\sT M$ can be identified with
$\sT_xM\op\sT_xM$, there are infinitely many possibilities to lift tensor fields from $M$ to $\sT M$. In particular, if the scalar product in $\sT_xM$ is $g$, we can take as a bilinear form on $\sT_{v_x}\sT M=\sT_xM\op\sT_xM$ any of the forms
$$\langle(v,v')\,|\,(w,w')\rangle=a\,g(v,w)+b\,g(v,w')+c\,g(v',w)+d\,g(v',w')\,.$$
The requirement that $\langle\cdot|\cdot\rangle$ is positively definite still offers infinitely many $a,b,c,d\in\R$ at our disposal. It is then clear that choosing one of them is not functorial. We get the Sasaki metric for $a=d=1$ and $b=c=0$. Another point is that the Sasaki construction does not work for lifts of metrics on arbitrary vector bundles (thus for Lie algebroid statistical structures considered in the next section), so it is useless for our purposes.
\end{remark}
\section{Lifting statistical structures on Lie algebroids}

\subsection{Homogeneity structures}
For a vector bundle $\zt:E\to M$ with the typical fiber $V$, \emph{affine coordinates} on $E$ associated with a local trivialization $\zt^{-1}(U)=U\ti V$ are $(x^a,y^i)$, where $x^a$ are local coordinates on $U\subset M$ and $y^i$ are linear coordinates in $V$, associated with a basis $e_*^i$ of $V^*$. They induce dual coordinates $(x^a,\zx_i)$ in the dual bundle $\zp:E^*\to M$, where $\zx_i$ are linear coordinates in the dual vector space $V^*$ associated with the basis $e_i$ dual to $e^i_*$.
In fact, there is a one-to-one correspondence $X\mapsto \zi_X$ between sections $X$ of $E$ and linear functions on $E^*$, given by $\zi_X(\zw)=\zP(X_{\zp(\zw)},\zw)$, where $\zP=\la\cdot,\cdot\ran:E\ti_ME^*\to\R$ is the canonical pairing,
$$\zP(x^a,y^i,\zx_j)=y^i\zx_i\,.$$
In affine coordinates, if $X=f^i(x)e_i$, then $\zi_X=f^i(x)\zx_i$. In particular, $\zx_i=\zi_{e_i}$ and $y^i=\zi_{e^i_*}$.

On any vector bundle $\zt:E\to M$ there is a canonical vector field $\n_E$ (called the \emph{Euler vector field}) being the generator
of the one-parametr group of diffeomorphisms $\R\ni t\mapsto h^E_{e^t}$ of $E$, where $h^E_s$ is the multiplication by $s$ in the vector bundle $E$ (the \emph{homogeneity structure} on $E$). In affine local coordinates, $h^E_s(x^a,y^i)=(x^a,sy^i)$ and
$\n_E=\sum_iy^i\,\pa_{y_i}$. A fundamental result in \cite{Grabowski:2009} states that any vector bundle is completely determined by its Euler vector field (or the homogeneity structure $h^E$). In this language, morphisms of vector bundles are just smooth maps intertwining the multiplications by reals in both vector bundles (or, equivalently, relating the corresponding Euler vector fields), and vector subbundles are just submanifolds which are invariant with respect to the multiplication by reals (or, equivalently, submanifolds to which the Euler vector field is tangent).

The Euler vector field determines naturally a general concept of homogeneity on $E$:  a tensor field $K$ on $E$ is \emph{homogeneous of weight $w\in\R$} if ${\Li}_{\n_E}(K)=w\,K$. In particular, affine coordinates are homogeneous, $x^a$ are of weight 0, and $y^i$ are of weight 1. Homogeneous functions of weight 1 are called also
\emph{linear}, and in affine coordinates $(x^a,y^i)$ take the form $\sum_if_i(x)y^i$. More generally,  \emph{linear $n$-forms} are $n$-forms of weight 1. Analogously, \emph{linear $n$-vector fields} are $n$-vector fields of weight $(n-1)$. In particular, an $n$-vector field $\zL$ on $E$ is linear if and only if the corresponding $n$-bracket of functions,
$$\{ g_1,\dots,\cdot g_n\}_\zL=\zL(\xd g_1,\dots,\xd g_n)\,,$$
is closed on linear functions.

\medskip\noindent Note that
$$h^E:\R\ti E\to E\,,\ h^E(s,\cdot)=h^E_s$$ is a particular example of a general \emph{homogeneity structure} \cite{Grabowski2012}, which is just a smooth action $h^F:\R\ti F\to F$ of the monoid $(\R,\cdot)$ of multiplicative reals on a manifold $F$ , i.e., $h^F_s\circ h^F_u=h^F_{su}$, $h^F_1=\id_F$. The corresponding analog of the Euler vector field is the \emph{weight vector field} $\n_F$. The fundamental result of \cite{Grabowski2012} says that with any homogeneity structure $h^F$ on $F$ there is associated a fiber bundle structure $\zt=h^F_0:F\to M=h^F_0(F)$ ($M$ is automatically a submanifold in $F$) with the typical fiber $\R^p$, $p\in\N$. Moreover, for a neighbourhood $U$ of any $p\in M$, there are homogeneous coordinates $(x^a,y^i)$ in $\zt^{-1}(U)$ such that $h^F_s(x^a,y^i)=(x^a,s^{w_i}y^i)$, where $w_i$ are positive integers. The corresponding weight vector field reads
$$\n_F=\sum_iw_i\cdot y^i\,\pa_{y^i}\,.$$
The transition maps for such systems of homogeneous coordinates are necessarily polynomial in variables $y^i$ \cite{Grabowski2012}, so the system of weights $(w_i)$ is uniquely determined globally up to permutations for a given homogeneity structure. Manifolds $F$ equipped with a homogeneity structure are called in \cite{Grabowski2012} \emph{graded bundles} (see also \cite{Bruce2016}), since they induce fiber bundle structures $\zt:F\to M$ . The largest $w_i$ is called the \emph{degree} of the homogeneity structure, so that vector bundles are just homogeneity structures of degree 1.
\begin{example}[\cite{Grabowski2012}]
Higher tangent bundles $\sT^rM$ are canonically graded bundles. The corresponding weight vector field in adapted coordinates $(x^i_\zl)$ reads
$$\n_{\sT^rM}=\sum_i\sum_{\zl=0}^r\zl\cdot x^i_\zl\,\pa_{x^i_\zl}=\sum_i\sum_{\zl=1}^r\zl\cdot x^i_\zl\,\pa_{x^i_\zl}\,.$$
In other words, the coordinate $x^i_\zl$ is homogeneous of weight $\zl$. The corresponding homogeneity structure
is $h^{\sT^rM}_s(x^i_\zl)=(s^\zl x^i_\zl)$. In a more intrinsic formulation (see \cite{Grabowski2012}), for a curve $\zg$ in $M$ we have
$$h^{\sT^rM}_s([\zg]_r)=[s.\zg]_r\,,\ \text{where}\ (s.\zg)(t)=\zg(st)\,.$$
\end{example}
\noindent The compatibility of various geometric structures with a given graded bundle (homogeneity structure) is considered in \cite{Grabowska2021}. Particular cases are \emph{weighted Lie algebroids} \cite{Bruce2016,Grabowska2021}
and \emph{weighted groupoids} \cite{Bruce2015,Grabowska2021}, being natural generalizations of the concepts of \emph{$VB$-algebroids} and \emph{$VB$-groupoids}, intensively studied recently in the literature (see e.g. \cite{Bursztyn2016,Mackenzie1992}).

\subsection{Higher lifts of vector bundles and their sections}\label{lvb}
An important observation is that homogeneity structures can be canonically lifted to higher tangent bundles \cite{Bruce2016,Grabowski2012}. More precisely, if $h^F$ is a homogeneity structure of degree $k$ on a graded bundle $F\to M$, then $\dtr h^F$, with $(\dtr h^{F})_s=\sT^rh^F_s$, is a homogeneity structure  of the same degree $k$ on the graded bundle $\sT^rF\to\sT^rM$. Moreover, the homogeneity structures $h^{\sT^rF}$ and $\dtr h^F$ are \emph{compatible}, i.e.
$$h^{\sT^rF}_s\circ (\dtr h^F)_u=(\dtr h^F)_u\circ h^{\sT^rF}_s\,.$$
Indeed, for a curve $\zg$ in $F$, we have $h^{\sT^rF}_s([\zg]_r)=[s.\zg]_r$, and for the action of the lifted homogeneity structure we have
$$(\dt^rh^F)_u([\zg]_r)=\sT^rh_u^F([\zg]_r)=[h_u^F\circ\zg]_r\,.$$
In consequence,
\be\label{compat}\left(h^{\sT^rF}_s\circ (\dtr h^F)_u\right)([\zg]_r)=\left((\dtr h^F)_u\circ h^{\sT^rF}_s\right)([\zg]_r)
=[h_u\circ\zg(st)]_r\,.\ee

\medskip\noindent
A manifold equipped with two compatible homogeneity structures we call a \emph{double graded bundle}. Double vector bundles (see \cite{Konieczna1999,Mackenzie1992,Pradines1974}) are particular examples, for which both graded bundle structures are of degree 1.
This implies that the higher tangent bundle $\sT^rE$ of a vector bundle $\zt:E\to M$ is canonically also a vector bundle $\sT^r\zt:\sT^rE\to\sT^rM$ with the homogeneity structure $\dtr h^E$. The corresponding Euler vector field is the complete lift $(\n_E)^c$ of the Euler vector field $\n_E$ on $E$. For affine coordinates $(x^a,y^i)$ on $E$ and the adapted coordinates $(x^a_\zl,y^i_\zn)$ on $\sT^rE$ we have (see \cite{Grabowska2021})
$$(\n_E)^c=\sum_{i,\zn}y^i_\zn\,\pa_{y^i_\zn}\,,$$
so the induced affine coordinates on $\sT^rE$ are $(x^a_\zl,y^i_\zn)$.
We obtain also a double graded bundle structure on $\sT^rE$ in which one structure is of degree 1, and the commutative diagram
\be\label{dgb}\xymatrix{
\sT^rE\ar[rr]^{\sT^r\zt} \ar[d]^{\zt^r_E} && \sT^rM\ar[d]^{\zt^r_M} \\
E\ar[rr]^{\zt} && M \,.}
\ee
The two compatible homogeneity structures are $h^{\sT^rE}$ and $\dt^rh^E$.
Similarly, starting from the dual affine coordinates $(x^a,\zx_i)$ in the dual bundle $\zp:E^*\to M$, we obtain the induced affine coordinates $(x^a_\zl,(\zx_i)_\zn)$ in the vector bundle $\sT^r\zp:\sT^rE^*\to\sT^rM$. An important observation is that we have a non-degenerate pairing between the  vector bundle structures of $\sT^rE$ and $\sT^rE^*$,
$$\zP^c=\la\cdot,\cdot\ran^{(r)}:\sT^rE\ti_{\sT^rM}\sT^rE^*\to\R\,,$$
represented by the complete lift of the canonical pairing
$\zP:E\ti_ME^*\to\R$. In local coordinates $(x^a,y^i,\zx_j)$ on $E\ti_ME^*$, the original pairing reads $\zP(x^a,y^i,\zx_j)=y^i\,\zx_i$, thus the the lifted pairing is (cf. (\ref{Lr}))
$$\zP^c(x^a_\zl, y^i_\zn,(\zx_j)_\zm)=y^i_\zn\cdot(\zx_i)_{r-\zn}\,.$$
In other words, the vector bundle $\sT^r\zp$ is dual to $\sT^r\zt$, and the coordinates dual to \\ $(y^i_0,\dots,y^j_1\dots,y^k_{r-1},y^l_r)$ are $\left((\zx_i)_r,(\zx_j)_{r-1},\dots,(\zx_k)_{1},(\zx_l)_0\right)$.

\medskip\noindent
For any vector bundle $\zt:E\to M$ we have canonical transformations (see \cite{Wamba2012})
$$\zq^{[\zl]}_E:\sT^rE\to\sT^rE\,,\quad \zq^{[\zl]}_E([\zg]_r)=[t^\zl\,\zg]_r\,,\quad \zl=0,1,\dots,r\,,
$$
where, for a curve $\zg:\R\to E$, the curve $(t^\zl\,\zg)$ is understood as $t\mapsto t^\zl\,\zg(t)$. The maps $\zq^{[\zl]}_\R:\sT^r\R\to\sT^r\R$ we will denote simply $\zq^{[\zl]}$. In particular, if $(x_\zn)$ are the adapted coordinates in $\sT^r\R\simeq \R^{r+1}$, then
$$x_\zm\left(\zq^{[\zl]}([\zg]_r)\right)=\frac{1}{\zm!}\frac{\xd^\zm}{\xd t^\zm}\,\bigg|_{t=0}(t^\zl\,\zg)=\frac{1}{\zm!}\dbinom{\zm}{\zl}\,\zl!\,\frac{\xd^{(\zm-\zl)}\zg}
{\xd\,t^{(\zm-\zl)}}(0)=x_{\zm-\zl}([\zg]_r)\,,$$
where $x_n=0$ for $n<0$, so that
$$\zq^{[\zl]}(x_0,\dots,x_r)=(0,\dots,0,x_0,\dots,x_{r-\zl})\,.$$
Moreover, the lifts of functions on $M$ to functions of $\sT^rM$ (cf. (\ref{fl})) can be described as
\be\label{flambda}f^{(\zl)}=\zq^{[r-\zl]}\circ\sT^r f\,.\ee
Indeed,
$$f^{(\zl)}([\zg]_r)=\frac{1}{\zl!}\left[\frac{\xd^{\zl}(f\circ\zg)}{\xd\,t^{\zl}}\right]_{t=0}
=\zq^{[r-\zl]}([f\circ\zg]_r)=\left(\zq^{[r-\zl]}\circ\sT^rf\right)([\zg]_r)\,.$$
This suggests that, for the sake of simplicity of notation, it could be reasonable to consider also $\zq^{(\zl)}_E=\zq^{[r-\zl]}_E$.

\medskip\noindent It is easy to see that the maps $\zq^{[\zl]}_E$ are morphisms of the double graded bundle (\ref{dgb}), i.e. they commute with $h^{\sT^rE}_s$ and $(\dtr h^E)_u$. Indeed, in view of (\ref{compat}), we have
$$\zg^{[\zl]}_E([u\,\zg(st)]_r)=[t^\zl\,u\,\zg(st)]_r\,,$$
which clearly does not depend on the order in which we apply $u,s$, and the multiplication by $t^\zl$. Since morphisms of vector bundles are smooth maps intertwining the corresponding multiplications by reals, it is also easy to see that for a morphism $\zF:E\to E'$ of vector bundles we have
\be\label{vbmchi}
\zq^{[\zl]}_{E'}\circ\sT^r\zF=\sT^r\zF\circ\zq^{[\zl]}_E\,.
\ee

\medskip\noindent
Let now $X:M\to E$ be a section of $E$. We define the sections $X^{[\zl]}$ and $X^{(\zl)}$ of the vector bundle $\sT^rE\to\sT^rM$ by (see \cite{Gancarzewicz1994,Wamba2012})
\be\label{ls}X^{[\zl]}=\zq^{[\zl]}_E\circ\sT^rX\,,\quad \zl=0,1,\dots,r\,,\ee
and $X^{(\zl)}=X^{[r-\zl]}$.
In other words,
$$X^{[\zl]}([\zg_M]_r)=[t^\zl\,X\circ\zg_M]_r\,.$$
In particular (see (\ref{flambda})), for a function on $M$ (i.e., a section of $M\ti\R\to M$) we have $f^{[\zl]}=f^{(r-\zl)}$, so that the new meaning of $f^{(\zl)}$ coincides with the one defined in Section \ref{s6}.
Another way of characterizing the lifts $X^{[\zl]}$ of sections of $E$ is the following.
\begin{proposition}
If $X$ is a section of a vector bundle $\zt:E\to M$, then $X^{[\zl]}$ is the unique section of the vector bundle $\sT^r\zt:\sT^rE\to\sT^rM$ such that
\be\label{iXl}\zi_{X^{[\zl]}}=(\zi_X)^{(r-\zl)}=(\zi_X)^{[\zl]}\,.\ee
Moreover, if $\zw$ is a section of $E^*$, then
\be\label{pairing} \langle X^{[\zl]},\zw^{[\zm]}\rangle=\zP^c\circ(X^{[\zl]},\zw^{[\zm]})=\langle X,\zw\rangle^{[\zl+\zm]}=\la X,\zw\rangle^{(r-\zl-\zm)}\,.
\ee
\end{proposition}
\begin{proof} Let $\zg:\R\to E^*$ and $\zg_M=\zp\circ\zg$. We have
\beas\displaystyle &\zi_{X^{[\zl]}}([\zg]_r)=\zP^c\left(X^{[\zl]}([\zg_M]_r),[\zg]_r\right)=
\zP^c\left([t^\zl\,X\circ\zg_M]_r,[\zg]_r)\right)=\zP^c\left([(t^\zl\,X\circ\zg_M,\zg)]_r\right)\\
&\displaystyle
=\frac{1}{r!}\frac{\xd^{r}}{\xd\,t^{r}}\,\bigg|_{t=0}\left(t^\zl\,\zi_X\circ\zg\right) =\frac{1}{r!}\dbinom{r}{\zl}\,\zl!\frac{\xd^{r-\zl}}{\xd\,t^{r-\zl}}\,\bigg|_{t=0}\left(\zi_X\circ\zg\right)
=(\zi_X)^{(r-\zl)}([\zg]_r)\,.
\eeas
\noindent As for the pairing, we have
\beas &\la X^{[\zl]},\zw^{[\zm]}\ran^{(r)}([\zg]_r)=\left(\zi_{X^{[\zl]}}\circ \zw^{[\zm]}\right)([\zg]_r)
=\left((\zi_X)^{[\zl]}\,([t^\zm\,\zw\circ\zg]_r)\right)=[t^\zl\,\zi_X(t^\zm\,\zw\circ\zg)]_r\\
&=[t^{\zl+\zm}\zi_X\circ\zw\circ\zg]_r=\la X,\zw\ran^{[\zl+\zm]}([\zg]_r)\,.
\eeas
\end{proof}
\noindent From the above proposition we easily get
\be\label{fX} (fX)^{[\zl]}=\sum_{\zm=0}^{r-\zl}f^{(\zm)}X^{[\zl+\zm]}\,.\ee
Moreover, if $X_1,\dots,X_k$ is a local basis of sections of $E$, and $\zw_1,\dots,\zw_k$ is the dual basis of sections of $E^*$, then the dual basis of $X_i^{[\zl]}$ is $\zw_i^{[r-\zl]}$.

\medskip\noindent The rule $$(T\otimes S)^{(\zl)}=\sum_{\zm=0}^\zl T^{(\zm)}\ot S^{(\zl-\zm)}\,,$$
applied already in Section \ref{s6}, will give us lifts of all $E$-tensor fields:
$$L^E_\zl:\mathscr{T}^q_p(E)\to\mathscr{T}^q_p(\sT^rE)\,,$$
where elements of $\mathscr{T}^q_p(E)$ are sections of the vector bundle $E^{\ot q}\ot(E^*)^{\ot p}$. The lifts $K^{(r)}$ we will call \emph{complete lifts} and denote $K^c$.
In view of (\ref{pairing}), we have a full analog of (\ref{fX}).
\begin{proposition}  If $K\in\mathscr{T}^q_p(E)$, then
\be\label{contraction1} K^{(\zl)}(X_1^{(\zm_1)},\dots,X_p^{(\zm_p)})=
\left(K(X_1,\dots,X_p)\right)^{(\zl+\zm-r\cdot p)}\,,
\ee
and
$$ K^{(\zl)}(\zw_1^{(\zn_1)},\dots,\zw_p^{(\zn_q)})=
\left(K(\zw_1,\dots,\zw_q)\right)^{(\zl+\zn-r\cdot q)}\,,
$$
where $\zm=\sum_i\zm_i$, $\zn=\sum_j\zn_j$, $X_1,\dots,X_p$ and $\zw_1,\dots\zw_q$ are sections of $E$ and $E^*$, respectively.
\end{proposition}

\begin{remark} If $E\to M$ is the tangent bundle $\zt_M:\sT M\to M$ (resp., the cotangent bundle $\zp_M:\sT^*M\to M$) and $X$ (resp., $\zw$) is its section, then we can
identify the lifted section $X^{[\zl]}:\sT^rM\to\sT^r\sT M$ (resp., $\zw^{[\zl]}:\sT^rM\to\sT^r\sT^* M$) with a vector field (resp., one-form) on $\sT^rM$, using the functor equivalences (\ref{fne})
$$\zk^r:\sT^r\circ\sT\to\sT\circ\sT^r\,,\quad \za^r:\sT^r\sT^*\to\sT^*\sT^r\,.$$
With these identifications, we will view $X^{[\zl]}$ as vector fields on $\sT^rM$ and $\zw^{[\zl]}$ as one-forms on $\sT^rM$. How they are related to the lifts of vector fields and one-forms defined in Section \ref{s6}?
We already know that the dual coordinates to fiber coordinates $(\dot x^i_\zm)$ on $\sT\sT^rM$ are $((p_i)_{r-\zm})$, so that the canonical symplectic form $\zw_{\sT^rM}$ on $\sT^*\sT^rM$ reads
$$\zw_{\sT^rM}=\xd (p_i)_\zm\we\xd x^i_{r-\zm}=(\zw_M)^c\,.$$
If $\{\cdot,\cdot\}_r$ is the corresponding Poisson bracket and $g:M\to\R$, then
$$X^{[\zl]}(g^{(\zm)})=\{\zi_{X^{[\zl]}},g^{(\zm)}\}_r=\{(\zi_X)^{(r-\zl)},g^{(\zm)}\}_r\,.$$
In particular, for $X=\pa_{x^j}$ and $g=x^i$, we get
$$(\pa_{x^j})^{[\zl]}(x^i_{\zm})=\{ (p_j)_{r-\zl},x^i_{\zm}\}_r=\zd^i_j\,\zd^{\zl}_\zm\,.$$
Hence, $(\pa_{x^j})^{[\zl]}=\pa_{x^i_\zl}=(\pa_{x^j})^{(r-\zl)}$, where we refer to the lifts $X^{(\zl)}$ introduced in Section \ref{s6}. In view of (\ref{fX}), in general $X^{[\zl]}=X^{(r-\zl)}$. The basis of sections dual to $((\pa_{x^i})^{[\zl]}=\pa_{x^i_\zl})$ is $(\xd x^i)^{[r-\zl]}$ (\ref{pairing}), so that $(\xd x^i)^{[\zl]}=\xd x^i_{r-\zl}$, and finally $\zw^{[\zl]}=\zw^{(r-\zl)}$. This shows that the lifts of sections $X^{(\zl)}=X^{[r-\zl]}$ coincide for vector fields and one-forms with the lifts defined in Section \ref{s6}.
In other words (cf. \cite{Wamba2012}), for a vector field $X$ and a one-form $\zw$ on $M$, we have
\be\label{nlifts}
X^{(\zl)}=\zk_M^r\circ\zq^{(\zl)}_{\sT M}\circ\sT^rX\,,\quad \zw^{(\zl)}=\ze_M^r\circ\zq_{\sT^*M}^{(\zl)}\circ\sT^r\zw\,.
\ee
Note, however, that our $\zq_E^{(\zl)}$ is $\zq_E^{(r-\zl)}$ in the notation of \cite{Wamba2012}.
\end{remark}

\subsection{Lie groupoids}
Let us first recall from \cite{Grabowska2019} what is a statistical structure on a Lie algebroid and what is a contrast function in this setting. For basic facts about Lie groupoids and Lie algebroids  we refer to \cite{Mackenzie2005,Meinrenken2017}.
\begin{definition} A \emph{Lie groupoid} is a manifold $G$ equipped with surjective	
submersions (target and source maps)  $\za,\zb:G\to M$ of $G$ onto a submanifold $M\subset G$ (we indicate this fact using the notation $G\rightrightarrows M$), and a partially defined associative product understood as a smooth map $m:G^{(2)}\to G$, where
$$G^{(2)}=\{(g,h)\in G\ti G\,|\, \zb(g)=\za(h)\}$$
(it is a smooth closed embedded submanifold of $G\ti G$). This means that the product $m(g,h)=gh$ is defined if and only if $\zb(g)=\za(h)$. We also assume that $g\zb(g)=g$ and $\za(h)h=h$, i.e., elements of $M$ are units for the multiplication. It is assumed additionally that there is a diffeomorphism (the inverse map),
$$\inv_G:G\to G\,,\ g\mapsto g^{-1}\,,$$
such that $gg^{-1}=\za(g)$ and $g^{-1}g=\zb(g)$.
By $\cF^\za(a)=\{ h\in G\,| \, \za(h)=a\}$ and $\cF^\zb(a)=\{ g\in G\,| \, \zb(g)=a\}$ we denote the target and source fibers ($\za$- and $\zb$-fibers), respectively. We define the left and right translations maps by
$$ l_g:\cF^\za(\zb(g))\to \cF^\za(\za(g))\,,\ l_g(h)=gh\,, \quad r_h:\cF^\zb(\za(h))\to\cF^\zb(\zb(h))\,,\ r_h(g)=gh\,.
$$
\end{definition}
\noindent Note that $l_g$ and $r_h$ are diffeomorphisms with the inverses $l_{g^{-1}}$ and $r_{h^{-1}}$, respectively. Writing a product $gh$ we automatically assume that $\zb(g)=\za(h)$. The foliations $\cF^\za$ and $\cF^\zb$ are associated with the involutive distributions $\ker(\sT\za)$ and $\ker(\sT\zb)$.

Of course, Lie groups are just Lie groupoids with $M$ being a single point (the unit). Another natural example is the \emph{pair groupoid} $G=M\ti M$. The corresponding target and source maps are the projections onto the first and onto the second factor, respectively, the multiplication is $(x,y)\bullet(y,z)=(x,z)$, the unit elements form the diagonal submanifold $\zD_M$, and the inverse is $\inv_G(x,y)=(y,x)$.

\subsection{Statistical Lie algebroids}
Like Lie algebras are the `infinitesimal parts' of Lie groups, the `infinitesimal parts' of Lie groupoids are \emph{Lie algebroids}.
There are many ways of defining Lie algebroids; we will mention two of them. The `standard' definition is the following.
\begin{definition} A \emph{Lie algebroid} is a vector bundle $ \tau: E\rightarrow M$ together with a Lie bracket $\left[.,. \right] $ on the space of its sections, such that there exists a vector bundle map ${\zr}:E\rightarrow\mathsf{T}M $ covering the identity on $M$, called the \emph{anchor map}, satisfying the Leibniz rule,
$$ \left[X, fY \right]=f\left[X,Y \right]+({\zr}(X) f) Y\,, $$ for all $X,Y\in\Sec(E)$,  $ f\in C^\infty(M)$.
\end{definition}
\begin{example} If $M$ is a manifold, then the tangent bundle $E=\sT M$ is canonically a Lie algebroid; the bracket of sections is the Lie bracket of vector fields, and the anchor is the identity map.
\end{example}

\begin{proposition}
\noindent A Lie algebroid structure on a vector bundle $E\to M$ is equivalent to fixing a linear Poisson tensor $\zL$ on $E^*$.
The Lie algebroid bracket $[\cdot,\cdot]$ and the anchor $\zr:E\to\sT M$ are related to the Poisson bracket $\{\cdot,\cdot\}_\zL$ associated with $\zL$ by
\be\label{vP}\zi_{[X,Y]}=\{\zi_X,\zi_Y\}_\zL\,,\quad
\zr(X)(f)\circ\zt=\{\zi_X,f\circ\zt\}_\zL\,.\ee
\end{proposition}
For Lie algebroids we can define the concepts of pseudo-Riemannian metrics, affine connections and their torsion, skewness tensors, etc., by an obvious analogy with the standard ones. We just replace vector fields (and the corresponding Lie bracket) with sections of $E$ (and the Lie algebroid bracket), their action on functions on $M$ by the action of sections \emph{via} the anchor map, $(p,q)$-tensors by sections of $E^{\ot p}\ot(E^*)^{\ot q}$, etc.

In particular,
a \emph{pseudo-Riemannian metric} on $E$ is a symmetric tensor field $g\in\Sec(E^\ast\ot E^*)$ such that $\Sec(E)\ni X \mapsto g(X,\cdot)\in\Sec(E^\ast)$ defines an isomorphism of vector bundles. In fact, this definition serves for all vector bundles, the Lie algebroid structure plays no r\^ole here. A vector bundle equipped with a pseudo-Riemannian metric we will call \emph{metric}. Analogously, a \emph{skewness tensor} on $E$ is a totally symmetric tensor $T\in\Sec((E^*)^{\ot 3})$.

\medskip\noindent  An \emph{(affine) connection ($E$-connection)} on  a Lie algebroid $E\to M$ is a $\R$-bilinear map (in fact, a differential operator)
$$ \nabla:\Sec(E)\times \Sec(E)\rightarrow \Sec(E)\,,\ (X,Y)\mapsto \nabla_X Y\,,$$
with the properties:
$$\nabla_{fX}Y = f\,\nabla_X Y\,, \ \nabla_X (fY) = f\nabla_X Y +{\zr}(X)(f)\,.$$
We will often call $E$-connections simply connections.
For the Lie algebroid $E=\mathsf{T}M$, connections on $E$ are standard affine connections.
Any connection $\n$ on $E$ can be extended to a covariant derivative acting on all $E$-tensor fields, as it is done in the classical situation.

\noindent The \emph{torsion} of a connection $\n$ on $E$ is the tensor field $\mathrm{Tor}^\nabla\in\Sec(\Lambda^2E^*\ot E)$ given by
$$\mathrm{Tor}^\nabla(X,Y)=\nabla_X Y-\nabla_Y X-[X,Y]\,.$$
Like in the standard differential geometry, for any metric Lie algebroid $(E,g)$ over $M$,
the dual connection $\n^*$ is defined completely analogously to the standard case,
$$g(\nabla^*_XY,Z)=\zr(X)\,(g(Y,Z))-g(Y,\nabla_XZ)\,.$$
One can prove that for any metric Lie algebroid $(E,g)$ there is a unique torsionless connection $\n^g$ (the \emph{Levi-Civita (or metric) connection}) such that $\n g=0$. There is a Koszul-like formula for $\n^g$, with the only difference with the standard one that sections $X$ of $E$ act on functions $f$ on $M$ \emph{via} the anchor map, $X(f)=\zr(X)(f)$.

There is a canonical one-to-one correspondence between torsionless connections $\n$ and skewness tensors $T$ on metric Lie algebroids given by
$$g\left(\nabla_XY,Z \right)=g\left(\nabla^g_XY,Z \right)-\frac{1}{2}T(X,Y,Z)\,.$$
In fact, like in the classical case we can obtain a one-parameter family of torsionless $E$-connections,
$$g\left(\nabla^\za_XY,Z \right)=g\left(\nabla^g_XY,Z \right)-\frac{\za}{2}T(X,Y,Z)\,,$$
where $\nabla^{-\za}=\left(\nabla^\za\right)^*$.
For metric Lie algebroids we have the full analogue of Theorem \ref{statm}.
All this suggests the following concept of a \emph{statistical Lie algebroid}.
\begin{definition}
Let $(E,g)$ be a metric Lie algebroid. A \emph{statistical structure} on $(E,g)$ is each of the following equivalent geometric structures:
\begin{itemize}
\item a skewness tensor $T\in\Sec\left((E^*)^{\ot 3}\right)$;
\item a torsionless Lie algebroid connection $\nabla$ such that $\nabla g$ is symmetric;
\item a torsionless Lie algebroid connection $\nabla$ such that $\nabla^g=\frac{1}{2}(\nabla+\nabla^*)$;
\item a pair $(\nabla,\nabla^*)$ of dual torsionless Lie algebroid connections.
\end{itemize}
\end{definition}
\noindent In the next section we explain how to understand contrast functions in the Lie algebroid setting.
\subsection{The Lie algebroid of a Lie groupoid}
To define the Lie algebroid $E=\Lie(G)\to M$, associated with a Lie groupoid $ G\rightrightarrows M$, as the vector bundle $\Lie(G)\to M$ we take the normal bundle,
$$\Lie(G)=E(M,G)=(\mathsf{T}G|_M)/\mathsf{T}M\,,$$
of $M$ in $G$. Let $X$ be a section of $\Lie(G)$ represented by a section $\bar X$ of $\sT G\,\big|_M$ (a vector field along $M$). It is easy to see that $(\sT\za-\sT\zb)(\bar X)$ does no depend on the choice of a representative $\bar X$, and so it is a vector field on $M$ which we denote $\zr(X)$. This defines a morphism $\zr:\Lie(G)\to\sT M$ which is the anchor of the Lie algebroid we are constructing.
A vector field $\tilde{X}$ on $G$ is called \emph{left-invariant} if it is tangent to the $\za$-fibers and invariant with respect to all left translations, $\sT l_g(\tilde X_h)=\tilde X_{gh}$. Similarly, $\tilde{X}$ is \emph{right-invariant} if it is tangent to the $\zb$-fibers and invariant with respect to the right translations.

\smallskip\noindent
By construction, the spaces $\mathfrak{X}^L(G)$ and $\mathfrak{X}^R(G)$ of left (resp., right) invariant fields on $G$ form Lie subalgebras in the Lie algebra of all vector fields. Since $\ker(\sT\za)|_M$ and $\ker(\sT\zb)|_M$ are complements to $ \mathsf{T}M $ in $ \mathsf{T}G|_M $, each of these bundles may be identified with the normal bundle $\Lie(G)$. Moreover,
it is easy to see that for any $X\in\Sec(\Lie(G))$ there is a unique left-invariant vector field $X^L\in\mathfrak{X}^L(G)$ and a unique right-invariant vector field  $X^R\in\mathfrak{X}^R(G)$ such that $X^L|_M$ and  $ X^R|_M$ represent $X$ in the normal bundle.
\begin{proposition}
For all $X\in\Sec(\Lie(G))$ we have
$$\sT\za(X^L)=0\,, \ \sT\zb(X^L)=\zr(X)\,, \ \sT\za(X^R)=-{\zr}(X)\,, \ \sT\zb(X^R)=0\,,\ \sT(\inv_G)(X^L)=-X^R\,.
$$
We have also $(fX)^L=f^LX^L\,, \  (fX)^R=f^RX^R\,,$ and
$$X^L(f^L)=({\zr}(X)f)^L\,,\  X^R(f^L)=0\,,\ X^L(f^R)=0\,,\  X^R(f^R)=-({\zr}(X)f)^R\,,$$
where $f^L=f\circ\zb$ and $f^R=f\circ\za$.
Moreover, there exists a unique Lie bracket $[\cdot,\cdot ]$ on $\Sec\left(\Lie(G)\right)$ such that (the reader will surely distinguish Lie algebroid brackets from the Lie bracket of vector fields)
$$[X^L,Y^L]=[X,Y]^L\,,\quad [X^R,Y^L]=0\,,\quad \text{and}\quad [X^R,Y^R]=-[X,Y]^R\,.
$$
\end{proposition}
\noindent The Lie algebroid structure on $\Lie(G)$ is determined by the anchor $\zr$ and the Lie bracket $[\cdot,\cdot ]$ defined above.

Note that not every Lie algebroid is of the form $\Lie(G)$ for a Lie groupoid $G$, but it is true if we allow for local Lie groupoids. This serves for our purposes, since it is enough to do differential calculus only in a neighbourhood of the submanifold $M$ of units in $G$.
\subsection{Contrast functions on Lie groupoids}
Our paper \cite{Grabowska2019} was motivated by the discovery that the concept of a contrast function on $M$ as a function on $M\ti M$, and its use in constructing statistical structures on $M$, refers only to the Lie groupoid structure on the pair groupoid $G=M\ti M$. The corresponding Lie algebroid $\Lie(M\ti M)$ is the canonical Lie algebroid $\sT M$ and the left (resp., right) invariant vector field corresponding to a vector field $X$ on $M$ is $X^L={ }^2\!X$ (resp., $X^R=-^1\!X$) in the notation of Section 4.

Let $G\rightrightarrows M$ be a (local) Lie groupoid, and $F:G\to\R$ be a contrast function on $(M,G)$ (cf. (\ref{Ncf})). According to
Remark \ref{rem}, the contrast function $F$ generates a (pseudo-Riemannian) metric $g^F$ on the normal bundle $E=\Lie(G)=E(M,G)$. This metric can be nicely described by means of invariant vector fields on $G$ (see \cite{Grabowska2019}). Namely, if $X,Y\in\Sec(\Lie(G))$, then
$$ g^F(X,Y)=X^LY^LF\,\big|_M=X^LY^RF\,\big|_M=X^RY^RF\,\big|_M\,.$$
Of course, if $F\ge 0$, then $g^F$ is positively defined. Moreover, $F$ is a contrast function on $(M,G)$ if and only if $F^*=F\circ\inv_G$ is a contrast function on $(M,G)$. In this case $g^F=g^{F^*}$. For the metric Lie algebroid $(\Lie(G),g^F)$, the rest of objects of the corresponding statistical Lie algebroid is defined as follows,
\beas & g(\nabla_X^F Y,Z)=X^LY^LZ^RF\,\big|_M\,,\quad  g(\nabla_X^{F^\ast} Y,Z)=X^LY^LZ^RF^*\,\big|_M\,, \\
& T^F(X,Y,Z)=X^R\,Y^L\,Z^L(F^*-F)\,\big|_M\,.
\eeas

\subsection{Lifting Lie groupoids and Lie algebroids}
It is well known that higher tangent bundles of Lie groupoids are canonically Lie groupoids themselves; we just apply the higher tangent functors. Let us start with a Lie groupoid $\za,\zb:G\rightrightarrows M$ with the multiplication $m:G^{(2)}\to G$ and the inverse map $\inv_G:G\to G$. The manifold $G^r=\sT^rG$ carries a canonical structure of a Lie groupoid which we will call the \emph{$r$-lift} of $G\rightrightarrows M$. We prefer the term `lift', since a \emph{prolongation of $G$} defined in the literature is another object.
For the lifted Lie groupoid $G^r$ the submanifold of units is $M^r=\sT^rM$, and the target and source maps are
$$\za^r=\sT^r\za\,,\,\zb^r=\sT^r\zb:\sT^rG\to\sT^rM\,.$$
Hence, $(\sT^rG)^{(2)}=\sT^r(G^{(2)})$ and the partial multiplication is
$\sT^rm:\sT^r(G^{(2)})\to\sT^rG$. Finally, the inverse map is $\inv^r_G=\inv_{G^r}=\sT^r(\inv_G)$. Consequently,
$$\Lie(G^r)=E(\sT^rM,\sT^rG)=\left(\sT\sT^rG\,\big|_{\sT^rM}\right)/\sT\sT^rM\simeq\left(\sT^r\sT G\,\big|_{\sT^rM}\right)/\sT^r\sT M=\sT^r\Lie(G)\,,$$
as a vector bundle over $M^r$.
Denoting the canonical projections $\sT^rG\to G$ and $\sT^rM\to M$ with $\zt^r_G:G^r\to G$ and $\zt^r_M:M^r\to M$, we get the commutative diagram
$$\xymatrix@C=30pt@R=40pt@=40pt{
G^r\ar[rr]^{\zt^r_G} \ar@<-.5ex>[d]_{\za^r}\ar@<.5ex>[d]^{\zb^r} && G\ar@<-.5ex>[d]_{\za}\ar@<.5ex>[d]^{\zb}  \\
M^r\ar[rr]^{\zt^r_M} && M \,.}
$$
If $g^r,h^r\in G^r$, then the left (resp., right) translations $l_{g^r}$ and $r_{g^r}$ act as
$$ l_{g^r}:\cF^{\za^r}(\zb^r(g^r))\to \cF^{\za^r}(\zb^r(g^r))\,, \quad r_{h^r}:\cF^{\zb^r}(\za^r(h^r))\to\cF^{\zb^r}(\zb^r(h^r))\,.
$$
Using the natural flow equivalence of functors, $\zk^r:\sT^r\circ\sT\to\sT\circ\sT^r$, we get the commutative diagram
$$\xymatrix@C=30pt@R=40pt@=40pt
{\sT^r\sT G\ar[rr]^{\zk^r_G} \ar@<-.5ex>[d]_{\sT^r\sT\za}\ar@<.5ex>[d]^{\sT^r\sT\zb} && \sT G^r\ar@<-.5ex>[d]_{\sT\za^r}\ar@<.5ex>[d]^{\sT\zb^r}  \\
\sT^r\sT M\ar[rr]^{\zk^r_M} && \sT M^r \,,}
$$
which means that $\sT\za^r\circ\zk^r_G=\zk^r_M\circ \sT^r\sT\za$ (and similarly for $\zb$). Therefore,
$$\ker(\sT\za^r)\,\big|_{M^r}=\zk^r_G\left(\sT^r\left(\ker(\sT\za)\,\big|_M\right)\right)\,.$$

\medskip\noindent It is obvious that a vector field $\cX$ on $G$ is tangent to $\za$-fibers, i.e. $\sT\za(\cX)=0$, if and only if $\cX(f^R)=0$ for all $f\in C^\infty(M)$, where $f^R=f\circ\za$. This is equivalent to the fact that the vector field $\widetilde\cX$ on $G\ti G$,
$$\widetilde\cX :G\ti G\to \sT(G\ti G)=\sT G\ti\sT G\,, \ \widetilde\cX(g,h)=(0_g,\cX_h)\,,$$
is tangent to the submanifold $G^{(2)}$, i.e., for $(g,h)\in G^{(2)}$ we have
$$\sT\za(\cX_h)=\sT\zb(0_g)=0_{\zb(g)}\,.$$
It is easy to see now that such a vector field $\cX$ is left-invariant if and only if the vector fields $\widetilde\cX$ and $\cX$ are $\sT m$-related, i.e.
$\sT m(\widetilde\cX(g,h))=\cX_{gh}$.
In other words,
$$\sT m\circ\widetilde\cX=\cX\circ m\,.$$
\begin{theorem}
If $X^L$ (resp., $X^R$) is the left-invariant (resp., the right-invariant) vector field on $G$ associated with a section $X:M\to\Lie(G)$, then $(X^L)^{(\zl)}$ (resp., $(X^R)^{(\zl)}$) is a left-invariant (resp., a right-invariant) vector field on $G^r$ associated with the lifted section (cf. (\ref{ls}))
$$X^{(\zl)}:\sT^rM\to\Lie(G^r)\,,\quad \zl=0,1,\dots,r\,.$$
\end{theorem}
\begin{proof}
Let $\cX$ be a left-invariant vector field on $G$. We will show that the lifts $\cX^{(\zl)}$ are left-invariant vector fields on $G^r$.  Note first that for any function $f$ on $M$ we have $(f^R)^{(\zn)}=(f^{(\zn)})^R$. Indeed (cf. (\ref{flambda})),
$$(f^R)^{(\zn)}=(f\circ\za)^{(\zn)}=\zq^{(\zn)}\circ\sT^r f\circ\sT^r\za=f^{(\zn)}\circ\sT^r\za=(f^{(\zn)})^R\,.$$
As the differentials of all lifts $f^{(\zn)}$ of functions on $M$ generate the cotangent bundle, it follows that a vector field on $G^r$ is tangent to $\za^r$-fibers if and only if it kills all functions $(f^R)^{(\zn)}$. But (cf. (\ref{vf}))
$$\cX^{(\zl)}((f^R)^{(\zn)})=(\cX(f^R))^{(\zl+\zn-r)}=0\,,$$
so the lifts of vector fields tangent to $\za$-fibers are tangent to $\za^r$-fibers. Observe also that
$$\wt{\cX^{(\zl)}}=(0,\cX^{(\zl)})=(0,\cX)^{(\zl)}=\wt{\cX}^{(\zl)}\,.$$
Suppose now that $\cX$ is left-invariant on $G$, so $\sT m\circ\wt{\cX}=\cX\circ m$. We know already that $\cX^{(\zl)}$ is tangent to $\za^r$-fibers, so it remains to show that
$$\sT\sT^rm\circ\wt{\cX}^{(\zl)}=\cX^{(\zl)}\circ\sT^rm\,.$$
Since (cf. (\ref{nlifts}))
$$\wt{\cX^{(\zl)}}=\wt{\cX}^{(\zl)}=\zk^r_{G^{(2)}}\circ\zq^{(\zl)}_{\sT G^{(2)}}\circ\sT^r\wt{\cX}$$
and $\sT\sT^rm\circ\zk^r_{G^{(2)}}=\zk^r_G\circ\sT^r\sT m$, we have (cf. (\ref{vbmchi}))
\beas &\sT\sT^rm\circ\wt{\cX^{(\zl)}}=\zk^r_G\circ\sT^r\sT m\circ\zq^{(\zl)}_{\sT G^{(2)}}\circ\sT^r\wt{\cX}
=\zk^r_G\circ\zq^{(\zl)}_{\sT G}\circ\sT^r\sT m\circ\sT^r\wt{\cX}\\
&=\zk^r_G\circ\zq^{(\zl)}_{\sT G}\circ\sT^r(\sT m\circ\wt{\cX})=
\zk^r_G\circ\zq^{(\zl)}_{\sT G}\circ\sT^r(\cX\circ m)=\zk^r_G\circ\zq^{(\zl)}_{\sT G}\circ\sT^r\cX\circ\sT^r m\\
&=\cX^{(\zl)}\circ\sT^rm\,,
\eeas
so $\cX^{(\zl)}$ is left-invariant.

What remains is the proof that, for any $X$ being a section of $E=\Lie(G)$, the left-invariant vector field $(X^L)^{(\zl)}$ equals $(X^{(\zl)})^L$, where $X^{(\zl)}$ is the lifted section.
Denote for simplicity $E=\Lie(G)$. We have a canonical embedding of vector bundles
$I_E:E\to\ker(\sT\za)\,\big|_{M}\subset \sT G$, so that sections $X:M\to E$ can be identified with $I_E\circ X$, which are vector fields on $G$ along the submanifold $M$. In other words, for any $x\in M$ the vector $X^L(x)$ coincides with $I_E(X_x)$. For the groupoid $G^r$ the Lie algebroid is $\sT^rE$, and the corresponding embedding $\sT^rE\to\sT\sT^rG$ reads $I_{\sT^r E}=\zk^r_G\circ\sT^rI_E$. Let us check how $(X^L)^{(\zl)}$ looks like at points of the submanifold $M^r$ of $G^r$. For, take a curve $\zg_M$ in $M$. We have
\beas &(X^L)^{(\zl)}([\zg_M]_r)=\left(\zk^r_G\circ\zq^{(\zl)}_{\sT G}\circ\sT^r(X^L)\right)([\zg_M)]_r)
=\left(\zk^r_G\circ\zq^{(\zl)}_{\sT G}\circ\sT^rI_E\circ\sT^rX)\right)([\zg_M]_r)\\
&=\left(\zk^r_G\circ\sT^rI_E\circ\zq^{(\zl)}_{E}\circ X)\right)([\zg_M]_r)=I_{\sT^rE}\circ X^{(\zl)}([\zg_M])\,,
\eeas
that finishes the proof.
\end{proof}
The above theorem immediately determines the Lie algebroid structure on $\Lie(G^r)=\sT^r\Lie(G)$. Let $X,Y$ be sections of $\Lie(G)$. By definition, the Lie bracket of sections $[X^{(\zl)},Y^{(\zm)}]$ is determined \emph{via} (we denote the Lie algebroid brackets and the brackets of vector fields with the same symbol, hoping that what bracket is in the play will be clear from the context)
$$ [X^{(\zl)},Y^{(\zm)}]^L=[\left(X^{(\zl)}\right)^L,\left(Y^{(\zm)}\right)^L]
=[(X^L)^{(\zl)},(Y^L)^{(\zm)}]=[X^L,Y^L]^{(\zl+\zm-r)}=\left([X,Y]^{(\zl+\zm-r)}\right)^L,$$
so that
\be\label{Lb}
[X^{(\zl)},Y^{(\zm)}]=[X,Y]^{(\zl+\zm-r)}\,.
\ee
Note that we have used (\ref{lb}) to write the last equality. The anchor map $\zr^r:\Lie(G^r)\to\sT M^r$ is defined on sections by
\beas & \zr^r(X^{(\zl)})([\zg_M]_r)=\sT\zb^r\left(\left(X^{(\zl)}\right)^L([\zg_M]_r)\right)=
\sT\zb^r\left((X^L)^{(\zl)}([\zg_M]_r)\right)\\
&=\left(\sT\sT^r\zb\circ\zk^r\circ\zq^{(\zl)}_{\sT G}\circ \sT^r(X^L)\right)([\zg_M]_r)
=\left(\zk^r\circ\sT^r\sT\zb\circ\zq^{(\zl)}_{\sT G}\circ \sT^r(X^L)\right)([\zg_M]_r)\\
&=\left(\zk^r\circ\zq^{(\zl)}_{\sT M}\circ\sT^r\sT\zb\circ \sT^r(X^L)\right)([\zg_M]_r)
=\left(\zk^r\circ\zq^{(\zl)}_{\sT M}\right)\left([\sT\zb\circ X^L\circ\zg_M]_r\right)\\
&=\left(\zk^r\circ\zq^{(\zl)}_{\sT M}\circ\zr(X)\right)([\zg_M]_r)=(\zr(X))^{(\zl)}([\zg_M]_r)\,,
\eeas
so that
\be\label{anch}
\zr^r(X^{(\zl)})=(\zr(X))^{(\zl)}\,.
\ee
Of course, we can take (\ref{Lb}) and (\ref{anch}) as the identities defining the $r$-lifted Lie algebroid structure on $\sT^r\zt:\sT^rE\to\sT^rM$, for an arbitrary Lie algebroid $\zt:E\to M$, not referring to any Lie groupoid (cf. \cite{Wamba2012}).
\begin{remark}
Another way of defining the Lie algebroid structure on $\sT^rE$ is by lifting the linear Poisson tensor $\zL$ on $E^*$ defining the Lie bracket and the anchor (cf. (\ref{vP})). Namely,
using the canonical identification $(\sT^rE)^*\simeq\sT^rE^*$, we get the linear Poisson bracket determining the Lie algebroid structure on $\sT^rE$ as $\zL^c$. The lifts $(\zi_X)^{(\zl)}$ of linear functions $\zi_X$ on $E^*$ correspond to the lifts $X^{(\zl)}$ of sections (cf. (\ref{iXl})), so that
$$\zi_{[X^{(\zl)},Y^{(\zm)}]}=\{(\zi_X)^{(\zl)},(\zi_Y)^{(\zm)}\}_{\zL^c}\,.$$
\end{remark}
\subsection{Lifting statistical Lie algebroids}
The lifts of tensor fields associated with a Lie algebroid $\zt:E\to M$ refer only to the lifted vector bundle structure, as described in Section \ref{lvb}. Therefore, if $g\in\Sec(E^*\ot E^*)$ is an $E$-metric and $T\in\Sec((E^*)^{\ot 3})$ is a skewness tensor on $E$, then the \emph{lifted statistical structure} on $\sT^rE$ is $(g^c,T^c)$. Mimicking the standard proof, we can see that the identity
$$g\left(\nabla^a_XY,Z \right)=g\left(\nabla^g_XY,Z \right)-\frac{a}{2}T(X,Y,Z)$$
defines a one-parameter family of torsionless $E$-connections $\n^a$ out of $(g,T)$ such that $(\n^a)^*=\n^{-a}$. The connection $\n^1$ we will denote simply $\n$ and call \emph{induced from the statistical structure} $(g^c,T^c)$. This justifies the following definition.
\begin{definition}
Let $(E,g,T)$ be a statistical Lie algebroid. Then, $(\sT^rE,g^c,T^c)$ is a statistical structure on the Lie algebroid $\sT^rE$, called the \emph{lifted Lie statistical structure} (or, more precisely, the \emph{$r$-lift} of the statistical structure $(E,g,T)$).
\end{definition}
On the other hand, if $\n$ is an $E$-connection, then by analogy to the standard case (cf. (\ref{connection})) one can prove that there exist a unique $\sT^rE$-connection $\n^c$ (called the \emph{complete lift} of $\n$) such that
\be\label{lc}\n^c_{X^{(\zl)}}Y^{(\zm)}=(\n_XY)^{(\zl+\zm-r)}\,.\ee
Moreover, if $\n$ is induced by the statistical structure $(E,g,T)$, then $\n^c$ is induced by the statistical structure $(\sT^rE,g^c,T^c)$. In this case only the anchored vector bundle structures $(E,\zr)$ and $(\sT^rE,\zr^r)$ are involved. The full Lie algebroid structures are needed for working with contrast functions, which is a more delicate question.

Of course, making use of the local formula (\ref{contrast1}), by mimicking the proof of Theorem \ref{thcontrast} we can prove that the complete lift $F^c$ of a contrast function $F:G\to\R$ on a Lie groupoid $G$ is a contrast function on the Lie groupoid $G^r$, and the following analog of Theorem \ref{main}.
\begin{theorem}
If $F:G\to\R$ is a contrast function on a Lie groupoid $G\rightrightarrows M$, then $$F^c:G^r=\sT^rG\to\R$$ is a contrast function on $G^r\rightrightarrows M^r$. Moreover, the metric $g^{F^c}$, the connection $\nabla^{F^c}$, its dual $\left(\nabla^{F^c}\right)^*$, and the skewness tensor $T^{F^c}$ induced by $F^c$ are the $r$-lifts of $g^F$, $\nabla^F$, $\left(\nabla^F\right)^*$, and $T^F$, respectively.
\end{theorem}
\begin{example}
Lut us put $\cG=\GL(n,\R)$ and let
$\cG\ti M\ni(\zg,x)\mapsto \zg x\in M$ be an action of the Lie group $\cG$ on a manifold $M$. For a given $\zg\in\cG$ the diffeomorphism $M\ni x\mapsto\zg x\in M$ we will denote $\zp(\zg)$. The corresponding \emph{action Lie groupoid} $G=\cG\ti M\rightrightarrows M$, with the submanifold of units $M\simeq \{ I\}\ti M$, is equipped with the target and source maps $\za(\zg,x)=\zg x$ and $\zb(\zg,x)=x$, the multiplication
$$(\zg,x)\bullet(\zg',x')=(\zg\zg',x')\,,$$
provided $\zg'x'=x$, and the inverse $\inv_G(\zg,x)=(\zg^{-1},\zg x)$.
Here,  $I\in\cG=\GL(n,\R)$ is the identity matrix.

\medskip\noindent
The corresponding \emph{action Lie algebroid} is defined on the vector bundle
$$E=\Lie(G)=\frak{g}\ti M\to M\,,$$
where $\frak{g}=\gl(n,\R)$ is the Lie algebra of $\cG$. Since this is a trivial vector bundle, we can view sections $\cX,\cY$ of $E$ as maps $\cX,\cY:M\to\g$, and understand the operations
$\cX\cY$, $\cX^t$ and $e^\cX$ on sections as the matrix multiplication, the matrix transpose, and the matrix exponential done fiberwise. In particular, we can identify $X\in\g$ with the constant map $X_E:M\to\g$, $X_E(x)=X$, and further with the constant section $\cX=X_E$ of $E$,
$$X_E:M\to E=\g\ti M\,,\quad X_E(x)=(X,x)\in\g\ti M\,.$$
It is obvious that such sections generate the vector bundle $E\to M$.

The anchor $\zr:E\to\sT M$ and the Lie bracket $[\cdot,\cdot]_E$ of sections of $E$ are uniquely determined by the conditions
$$\zr(X_E)=\hat X\,,\quad [X_E,Y_E]_E=\left([X,Y]\right)_E \,,$$
where $\hat X$ is the fundamental vector field of the $\cG$-action on $M$, associated with $X\in\frak{g}$. In other words, $\hat X$ is the generator of the one-parameter group of diffeomorphisms $t\mapsto \zp(\exp(-tX))$, so that we have
\be\label{fa}
\widehat{[X,Y]}=[\hat X,\hat Y]\,,
\ee
where, clearly, the first bracket is the Lie bracket on $\g$, and the second is the bracket of vector fields on $M$.
In particular, for $X,Y\in\g$,
$$\zr\left(fX_E\right)=f\hat X\,,\ \text{and}\quad [fX_E,gY_E]_E=fg[X,Y]_E+f\hat X(g)Y_E-g\,\hat Y(f)X_E\,.$$
\begin{remark}
Our definition of the fundamental vector field $\hat X$ uses $(-t)$ as the parameter, in order to have the Lie algebra homomorphism (\ref{fa}). With the one-parameter group of diffeomorphism $\zp(\exp(tX))$ we would have an anti-homomorphism. For example, the left-regular action of a Lie group on itself would produce right-invariant vector fields as fundamental vector fields. This is a mistake which is often found in the literature. This is probably because in applications to principal bundles one uses the right action of the group, which is in fact an anti-action.
\end{remark}
Since in our case $\cG=\GL(n,\R)$ is an open subset in the vector space $\g=\gl(n,\R)$, we can consider $\sT\cG$ as an open subset in $\g\ti\g$ consisting of
pairs $(\zg,X)\in\g\ti\g$ such that $\zg$ is invertible. In other words, we have the identification
$\sT\cG\simeq\cG\ti\g$, where $(\zg,X)$ represents the first jet of the curve $t\mapsto\zg+tX$ (with values in $\cG$ for small $t$). This vector $(\zg,X)$ corresponds to the vector $(\zg,\zg^{-1}X)$ (resp., $(\zg,X\zg^{-1})$) in the left (resp., right) trivialization of $\sT\cG$.
The latter trivializations are valid for an arbitrary Lie group, while in our case it is important that $\cG=\GL(n,\R)$ is canonically an open subset in a vector space. On the other hand, our trivialization is easier to work with; in particular, the left translation of $X\in\g$ by $\sT(l_\zg)$ reads simply $\zg X$, where the product is the matrix multiplication.
The vector fields on $\cG$ are viewed as maps $\cX:\cG\to\g$, canonically identified with the sections $\cG\ni\zg\mapsto (\zg,\cX(\zg))\in\sT\cG$.
The left- and right-invariant vector fields on $\cG$ corresponding to $X\in\g$ are in this picture
$$X^L(\zg)=\zg\, X\quad\text{and}\quad X^R(\zg)=X\zg\,.$$
What are the invariant vector fields on the Lie groupoid $G=\cG\ti M$ corresponding to $X_E$? It is a matter of direct calculations to get
\be\label{ivf} (X_E)^L=(X^L,\hat X)\quad\text{and}\quad(X_E)^R=(X^R,0)\,,\ee
where $(X^L,\hat X)$ and $(X^R,0)$ have clear meaning as vector fields on $G=\cG\ti M$.

Consider now the function
\be\label{cf}F_0:\cG\to\R\,,\quad F_0(\zg)=\tr\big[(I-{\zg})(I-{\zg}^t)\big]\,,\ee
extended trivially to the function $F$ on the Lie groupoid $G$,
$$F(\zg,x)=F_0(\zg)=\tr\big[(I-{\zg})(I-{\zg}^t)\big]\,,$$
which clearly vanishes on $M\simeq\{ I\}\ti M$.  Here, `$t$' denotes the matrix transposition.
Since $F$ does not depend on points of $M$, (\ref{ivf}) implies that taking derivatives of $F$ with respect to the invariant vector fields $(X_E)^L$, $(X_E)^R$, etc., and then restricting them to $M\simeq\{ I\}\ti M$, will produce constant functions on $M$ with values being the derivatives of $F_0$ with respect to $X^L$, $X^R$, etc., evaluated at $I$.
In particular, for $X,Y\in\g$, we have
$$(X_E)^L(Y_E)^RF\,\big|_M=(X^LY^RF_0)(I)\,,$$
so that the symmetric two-tensor (metric) associated with $F$ reads
\beas & g^F(X_E,Y_E)=(X_E)^L(Y_E)^RF\,\big|_M=(X^LY^RF_0)(I)\\
&\displaystyle =\frac{\pa^2}{\pa u\,\pa s}\,\bigg|_{u,s=0}\!\tr\big[(I-e^{uY}\,e^{sX})(I-e^{sX^t}e^{uY^t})\big]\,.
\eeas
By straightforward calculations we get
\be\label{metr}g^F(X_E,Y_E)=2\,\tr(X\,Y^t)\,.\ee
This is a positive definite metric on $E=\gl(n,\R)\ti M$, so that $F$ is a Lie contrast function.
Similarly, for the torsionless $E$-connection $\n^F$ we have
$$g^F(\nabla^F_{X_E}\,Y_E,Z_E)=-(X_E)^R(Y_E)^R(Z_E)^L\,F\,\big|_M\,,$$
which equals
$$-X^RY^RZ^L\,F_0(I)=-\frac{\pa^3}{\pa v\,\pa u\,\pa s}\,\bigg|_{v,u,s=0}\,F_0\left(e^{sZ}\,e^{vX}\,e^{uY}\right)\,.$$
Tedious but easy calculations lead to
$$g^F(\nabla^F_{X_E}\,Y_E,Z_E)=-2\tr\big[\left(X^tY+Y^tX+YX\right)Z^t\big]\,.$$
To get it, we used the well-known properties of the trace: $\tr(A)=\tr(A^t)$ and $\tr(AB)=\tr(BA)$.
From (\ref{metr}) we get now
\be\label{Econ}\n^F_{X_E}Y_E=-\left(X^tY+Y^tX+YX\right)_E\,.\ee
The connection is indeed torsionless:
$$\n^F_{X_E}Y_E-\n^F_{Y_E}X_E=-(YX-XY)_E=[X_E,Y_E]_E\,.$$
We have analogously
$$\n^{F^*}_XY=-\left(XY^t+YX^t+XY\right)_E\,,$$
so that the Levi-Civita connection for $g^F$,
$$\n^{g^F}=\frac{1}{2}\left(\n^F+\n^{F^*}\right)\,,$$
reads
$$\n^{g^F}_{X_E}Y_E=
\frac{1}{2}\left([X^t,Y]+[Y^t,X]+[X,Y]\right)_E\,.$$
Similarly, we get the corresponding skewness tensor,
$$T(X_E,Y_E,Z_E)=(X_E)^R(Y_E)^L(Z_E)^L(F^*-F)\,\big|_M
=(X^R\,Y^L\,Z^L)(F_0^*-F_0)(I)\,.$$
In the explicit form,
\be\label{T1}T(X_E,Y_E,Z_E)=-2\,\tr\big[(XY+YX)Z^t+(XZ+ZX)Y^t+(ZY+YZ)X^t\big]\,.\ee

\medskip\noindent
In order to lift the above Lie statistical structure and the corresponding contrast function to the tangent bundles, let us recall first that for a Lie group $\cG$ and its Lie algebra $\g$ the multiplication in $\sT \cG$ is (see e.g. \cite[10.17]{Kolar1996} or \cite[Chapter V]{Yano1973})
$$X_g\bullet Y_h=\sT_h(l_g)(Y_h)+\sT_g(r_h)(X_g)\,.$$
In particular, the unit in $\sT\cG$ is $0_e$, where $e$ is the unit in $\cG$.
For our trivialization $\sT\cG=\cG\ti\g$ in the case $\cG=\GL(n,r)$, the multiplication in $\sT\cG$ reads
$$(\zg,X)\bullet(\zg',X')=(\zg\zg',\zg X'+X\zg')\,.$$
The Lie bracket for the tangent Lie algebra $\sT\g\simeq\g\op\g_0$ ($\g_0$ is another copy of $\g$) is
$$[(X,X_0),(Y,Y_0)]_{\sT\g}=\left([X,Y],[X,Y_0]+[X_0,Y]\right)\,.$$
This gives also the Lie bracket for the Lie algebroid $\sT E$ of constant sections $(X,X_0)_{\sT E}$ and $(Y,Y_0)_{\sT E}$.
Since the considered vector bundles are trivial, it is easy to see that the morphisms $\zq_E^{(\zl)}:\sT E\to\sT E$, $\zl=0,1$, reduce to the morphisms $\zq^{(\zl)}_{\sT\g}:\sT\g\to\sT\g$ in fibers,
$$\zq^{(1)}_{\sT\g}=\id_{\sT\g}\quad\text{and}\quad \zq^{(0)}_{\sT\g}(X,X_0)=(0,X)\,.$$
Hence,
\be\label{lf}(X_E)^c:=(X_E)^{(1)}=(X,0)_{\sT E}\quad\text{and}\quad (X_E)^v:=(X_E)^{(0)}=(0,X)_{\sT E}\,.\ee
Moreover, the anchor for the Lie algebroid $\sT E=\sT\g\ti\sT M$ reads
$$\zr^1\left((X,Y)_{\sT E}\right)=(\hat X)^c+(\hat Y)^v\,.$$
In view of (\ref{cf}), (\ref{metr}), and (\ref{T1}), after calculating the `time derivative' (we can also make use of (\ref{contraction1}), (\ref{metr}), (\ref{T1}), and (\ref{lf})), we easily obtain the forms of the complete (tangent) lifts of $F$, $g^F$, and $T^F$:
\beas & F^c(\zg,X,x,\dot x)=\tr\big[\zg X^t+X\zg^t-(X+X^t)\big]\,,\\
&\left(g^F\right)^c\big((X,X_0)_{\sT E},(Y,Y_0)_{\sT E}\big)=2\tr\big[XY_0^t+X_0^tY\big]\,,
\eeas
and
\beas
&&\left(T^F\right)^c\big((X,X_0)_{\sT E},(Y,Y_0)_{\sT E},(Z,Z_0)_{\sT E}\big)\\
&=&-2\tr\bigg[\big[(X_0Y+XY_0)Z^t+XYZ_0^t\big]\,+\,\text(perm.)\bigg]\,,
\eeas
where `(perm.)' denotes all permutations of symbols $(X,Y,Z)$.
To calculate $\left(\n^F\right)^c$,
we use (\ref{Econ}) and (\ref{lc}) to get
\beas & \left(\n^F\right)^c_{(X,X_0)_{\sT E}}(Y,Y_0)_{\sT E}\\
&=-\bigg(\big(X^tY+Y^tX+YX\big)\,,
\big(X_0^tY+Y^tX_0+YX_0+X^tY_0+Y_0^tX+Y_0X\big)\bigg)_{\sT E}\,.
\eeas
Now, it is easy to check that this connection is torsionless,
\beas &\left(\n^F\right)^c_{(X,X_0)_{\sT E}}(Y,Y_0)_{\sT E}-\left(\n^F\right)^c_{(Y,Y_0)_{\sT E}}(X,X_0)_{\sT E}
&=\big([X,Y]\,,[X_0,Y]+[X,Y_0]\big)_{\sT E}\\
&=\big[(X,X_0)_{\sT E}\,,(Y,Y_0)_{\sT E}\big]_{\sT E}\,.
\eeas
\end{example}

\section{Conclusions and outlook}
We have shown that the introduced lifting procedures provide
new examples of statistical manifolds, and that the lifted metrics, skewness tensors, and the connections may be derived from the lifted contrast functions. Our
construction is functorial and shows that to achieve this result we have
to admit metric tensors which are not positive definite. Although
the meaning of `null-varieties' of probability distributions at the
moment is unclear to us, we hope to come back to this aspect very soon.

As one possible application of our construction in the quantum setting,
we mention the evolution of quantum states for open systems.
It is well known that when we consider a quantum system coupled to an
environment, the  evolution of the coupled system defines projected
trajectories on the space of states of the subsystem, which in general
are not trajectories of a vector field on the space of the quantum states of
the subsystem. Indeed, this happens only under particular assumptions of a
weak coupling, and when it happens the trajectories are solutions of a
vector field which generates a semi-group of completely positive maps (the
infinitesimal generator being a GKLS-vector field).

When the family of trajectories is not associated with a vector field, one
speaks of non-Markovian evolution. We believe that by considering
differential equations of higher order it may be possible to describe
families of projected trajectories as solutions of first order vector
fields on the higher tangent bundles we have introduced.
A detailed discussion of the construction of a vector field
on $\sT M$ out of trajectories on $M$ is to be found in \cite{Marmo1985}.
More likely, this approach will not be able to give all systems with memory, but we may be
able to capture a relevant part of them. These aspects will appear
elsewhere.

One of important discoveries is also the fact that the above lifting procedures can be naturally extended to the categories of Lie algebroids and Lie groupoids. In this setting, we replace contrast functions defined on the manifold $M\ti M$ (which is actually a Lie groupoid) with contrast functions on an arbitrary Lie groupoid. To obtain the Lie algebroid tensors and connections, we differentiate the contrast function by left- and right-invariant vector fields on the Lie groupoid.
The lifts in this case lead to Lie algebroid and Lie groupoid structures on the corresponding higher tangent bundles we have studied in detail.

\section*{Acknowledgments}
G.~Marmo is a member of the Gruppo Nazionale di Fisica Matematica (INDAM), Italy. He acknowledges financial support from the Spanish Ministry of Economy and Competitiveness, through the Severo Ochoa Program for Centers of Excellence in RD (SEV-2015/0554). M.~Ku\'s acknowledges support of the Polish National Science Center \emph{via} the OPUS grant 2017/27/B/ST2/02959.  The authors thank E.~Peyghan for turning their attention to paper \cite{Matsuzoe2003}.

\section*{Conflict of Interest}
The authors declare that they have no conflict of interest.

\vskip.5cm
\noindent Katarzyna Grabowska\\\emph{Faculty of Physics,
University of Warsaw,}\\
{\small ul. Pasteura 5, 02-093 Warszawa, Poland} \\{\tt konieczn@fuw.edu.pl}\\
https://orcid.org/0000-0003-2805-1849\\

\noindent Janusz Grabowski\\\emph{Institute of Mathematics, Polish Academy of Sciences}\\{\small ul. \'Sniadeckich 8, 00-656 Warszawa,
Poland}\\{\tt jagrab@impan.pl}\\  https://orcid.org/0000-0001-8715-2370
\\

\noindent Marek Ku\'s\\
\emph{Center for Theoretical Physics, Polish Academy of Sciences,} \\
{\small Aleja Lotnik{\'o}w 32/46, 02-668 Warszawa,
Poland} \\{\tt marek.kus@cft.edu.pl}\\
https://orcid.org/0000-0002-2767-3251
\\

\noindent Giuseppe Marmo\\
\emph{Dipartimento di Fisica ``Ettore Pancini'', Universit\`{a} ``Federico II'' di Napoli} \\
\emph{and Istituto Nazionale di Fisica Nucleare, Sezione di Napoli,} \\
{\small Complesso Universitario di Monte Sant Angelo,} \\
{\small Via Cintia, I-80126 Napoli, Italy} \\
{\tt marmo@na.infn.it}\\
https://orcid.org/0000-0003-2662-2193
\\

\begin{thebibliography}{10}
\bibitem{Amari1982}
S.-i.~Amari.
\newblock{ Differential Geometry of Curved Exponential Families Curvature and Information Loss.}
\newblock{\emph{Ann. Statist.} \textbf{10}, 357--385, 1982.}

\bibitem{Amari1985}
S.-i.~Amari.
\newblock{Differential geometric methods in statistics.}
\newblock{\emph{Lect. Notes in Statist.} \textbf{28}, Springer, Heidelberg, 1985.}

\bibitem{Amari1987}
S.-i.~Amari et al.
\newblock{Differential Geometry in Statistical Inference.}
\newblock{\emph{IMS Lecture Notes--Monograph Series} \textbf{10},
Institute of Mathematical Statistics, Hayward, California, 1987.}

\bibitem{Amari2007}
S.-i.~Amari and H.~Nagaoka.
\newblock{Methods of information geometry.}
\newblock{American Math. Soc., Providence, R. I., 2007.}
	
\bibitem{Amari2012}
S.-i. Amari.
\newblock{Differential-geometrical methods in statistics.}
\newblock{Springer, Berlin/Heidelberg, 2012.}
	
\bibitem{Amari2016}
S.-i. Amari.
\newblock{Information geometry and its applications.}
\newblock{Springer Japan, Berlin/Heidelberg, 2016.}
	
	

	


\bibitem{Barndorff1986}
O.~E.~Barndorff-Nielsen, D.~R.~Cox and N.~Reid.
\newblock{The Role of Differential Geometry in Statistical Theory.}
\newblock{\emph{International Statistical Review} \textbf{54}, 83--96, 1986.}

\bibitem{Barndorff1997}
O.~E.~Barndorff-Nielsen and P.~E.~Jupp.
\newblock{Statistics, yokes and symplectic geometry.}
\newblock{\emph{Ann. Fac. Sci. Toulouse Math.}, \textbf{6} (3), 389--427, 1997.}

\bibitem{Blesild1991}
P.~Bl{\ae}sild.
\newblock{Yokes and tensors derived from yokes.}
\newblock{\emph{Ann. Inst. Statist. Math.}, \textbf{43}, 95--113, 1991.}

\bibitem{Bruce2015} A.~J.~Bruce, K.~Grabowska and J.~Grabowski.
\newblock{Graded bundles in the category of Lie groupoids.}
\newblock{\emph{SIGMA Symmetry Integrability Geom. Methods Appl.} \textbf{11} (2015), Paper 090, 25 pp.}

\bibitem{Bruce2016} A.~J.~Bruce, K.~Grabowska and J.~Grabowski.
\newblock{Linear duals of graded bundles and higher analogues of (Lie) algebroids.}
\newblock{\emph{J. Geom. Phys.} \textbf{101 }, 71--99, 2016.}

\bibitem{Bursztyn2016} H.~Bursztyn, A.~Cabrera and M.~del Hoyo.
\newblock{Vector bundles over Lie groupoids and algebroids.}
\newblock{\emph{Adv. Math.} \textbf{290} (2016), 163--207.}

\bibitem{Combe2021} N.~C.~Combe, Y.~I.~Manin, M.~Marcolli.
\newblock{Geometry of information: Classical and quantum aspects.}
\newblock{\emph{Theoretical Computer Science}, online.}

\bibitem{Calin2014}
O.~Calin and C.~Udri\c{s}te.
\newblock{Geometric modeling in probability and statistics.}
\newblock{Springer, Cham, 2014.}

\bibitem{Cantrijn1989} F.~Cantrijn, M.~Crampin, W.~Sarlet and D.~Saunders.
\newblock{The canonical isomorphism between $\sT^k\sT^*$ and $\sT^*\sT^k$.}
\newblock{\emph{C. R. Acad. Sci. Paris} \textbf{309}, S\'erie II, 1509--1514, 1989.}
    	

\bibitem{Censov1982}
N.~N.~Chentsov.
\newblock{Statistical decision rules and optimal inferences.}
\newblock{\emph{Transl. Math. Monogr.} \textbf{53}, 117--120, 1982.}
	
\bibitem{Ciaglia2017}
F.~M.~Ciaglia, F.~Di~Cosmo, D.~Felice, S.~Mancini, G.~Marmo, and J.~M.~P\'erez-Pardo.
\newblock{Hamilton-{J}acobi approach to potential functions in information geometry.}
\newblock{\emph{Journal of Mathematical Physics}, \textbf{58}, 063506, 2017.}
	
\bibitem{Ciaglia2018}
F.~M.~Ciaglia, F.~Di Cosmo, M.~Laudato, G.~Marmo, F.~M.~Mele, F.~Ventriglia, and P.~Vitale.
\newblock{A pedagogical intrinsic approach to relative entropies as potential functions of quantum metrics: The q–z family.}
\newblock{\emph{Ann. Phys.}, \textbf{395}, 238 -- 274, 2018.}
	
\bibitem{Ciaglia2019}
F.~M. Ciaglia, G.~Marmo, and J.~M. P\'erez-Pardo.
\newblock{Generalized potential functions in differential geometry and information geometry.}
\newblock{\emph{Int. J. Geom. Methods Mod. Phys.} \textbf{16} (supp01), 1940002, 2019.}

\bibitem{Dombrowski1962}
P.~Dombrowski.
\newblock{On the geometry of the tangent bundle.}
\newblock{\emph{J. Reine Angew. Math.} \textbf{210}, 73--88, 1962.}

\bibitem{Euguchi1985}
S.~Eguchi.
\newblock{A differential geometric approach to statistical inference on the basis of contrast functionals.}
\newblock{\emph{Hiroshima Math. J.} \textbf{15}, 341--391, 1985.}

\bibitem{Euguchi1992}
S.~Eguchi.
\newblock{Geometry of minimum contrast.}
\newblock{\emph{Hiroshima Math. J.} \textbf{22}, 631--647, 1992.}

\bibitem{Facchi2010}
P.~Facchi, R.~Kulkarni, V.~I.~Man'ko, G.~Marmo, E.~C.~G.~Sudarshan, and F.~Ventriglia.
\newblock{Classical and quantum Fisher information in the geometrical formulation of quantum mechanics.}
\newblock{\emph{Phys. Lett. A} \textbf{374}(48), 4801--4803, 2010.}


\bibitem{Gancarzewicz1994} J.~Gancarzewicz, W.~Mikulski and Z.~Pogoda.
\newblock{Lifts of some tensor fields and connections to product preserving functors.}
\newblock{\emph{Nagoya Math. J.} \textbf{135} (1994), 1--41.}

\bibitem{Grabowska2019}	K.~Grabowska, J.~Grabowski, M.~Ku\'s, and G.~Marmo.
\newblock{Lie groupoids in information geometry.}
\newblock{\emph{J. Phys. A} \textbf{52}, 505202, 2019.}

\bibitem{Grabowska:2008} K.~Grabowska and J.~Grabowski.
\newblock{Variational calculus with constraints on general algebroids.}
\newblock{\emph{J. Phys. A} \textbf{41} (2008), no. 17, 175204, 25 pp.}

\bibitem{Grabowska:2011} K.~Grabowska and J.~Grabowski.
\newblock{Dirac algebroids in Lagrangian and Hamiltonian mechanics.}
\newblock{\emph{J. Geom. Phys.} \textbf{61} (2011), no. 11, 2233--2253.}

\bibitem{Grabowska:2006} K.~Grabowska, P.~Urba\'nski and J.~Grabowski.
\newblock{Geometrical mechanics on algebroids,}
\newblock{\emph{Int. J. Geom. Methods Mod. Phys.} \textbf{3} (2006), no. 3, 559--575.}

\bibitem{Grabowska2020}	K.~Grabowska, J.~Grabowski, M.~Ku\'s and G.~Marmo.
\newblock{Information geometry on groupoids: the case of singular metrics.}
\newblock{\emph{Open Syst. Inf. Dyn.} \textbf{27}, 2050015, 2020.}

\bibitem{Grabowska2021} K.~Grabowska, J.~Grabowski and Z.~Ravanpak.
\newblock{VB-structures and generalizations.}
\newblock{\emph{Ann. Global Anal. Geom.} \textbf{62} (2022), 235--284.}

\bibitem{Grabowski:2009} J.~Grabowski and M.~Rotkiewicz.
\newblock{Higher vector bundles and multi-graded symplectic manifolds.}
\newblock{\emph{J. Geom. Phys.} \textbf{59}, 1285--1305, 2009.}

\bibitem{Grabowski2012} J.~Grabowski and M.~Rotkiewicz.
\newblock{Graded bundles and homogeneity structures,}
\newblock{\emph{J. Geom. Phys.} \textbf{62}, 21--36, 2012.}

\bibitem{Grabowski1995} J.~Grabowski and P.~Urba\'nski.
\newblock{Tangent lifts of Poisson and related structures.}
\newblock{\emph{J. Phys. A} \textbf{28} (1995), no. 23, 6743--6777.}

\bibitem{Grabowski1997} J.~Grabowski and P.~Urba\'nski.
\newblock{Tangent and cotangent lifts and graded Lie algebras associated with Lie algebroids.}
\newblock{\emph{Ann. Global Anal. Geom.} \textbf{15} (1997), no. 5, 447--486.}

\bibitem{Grabowski1999} J.~Grabowski and P.~Urba\'nski.
\newblock{Algebroids--general differential calculi on vector bundles.}
\newblock{\emph{J. Geom. Phys.} \textbf{31} (1999), no. 2-3, 111--141.}

\bibitem{Henmi2011a}
M.~Henmi and H.~Matsuzoe.
\newblock{Geometry of pre-contrast functions and non-conservative estimating functions.}
\newblock{in: \emph{AIP Conf. Proc.} \textbf{1340} (2011), 32--41.}

\bibitem{Henmi2011}
M.~Henmi, and H.~Matsuzoe.
\newblock{Statistical manifolds admitting torsion and partially flat spaces.}
\newblock{in: \emph{Geometric Structures of Information} (Springer, Cham, 2019), pp. 37--50.}

\bibitem{Kolar1996a} I.~Kol\'a\v{r}.
\newblock{Natural operations on higher order tangent bundles.}
\newblock{(Proceedings of the Third Meeting on Current Ideas in Mechanics and Related Fields, Segovia, 1995), \emph{Extracta Math.} \textbf{11} (1996), 106--115.}

\bibitem{Kolar1996} I.~Kol\'a\v{r}, P.~W.~Michor and J.~Slov\'ak.
\newblock{Natural operations in differential geometry.}
\newblock{Springer-Verlag, Berlin, 1993.}

\bibitem{Konieczna1999} K.~Konieczna and  P.~Urba\'nski.
\newblock{Double vector bundles and duality.}
\newblock{\emph{Arch. Math. (Brno)} \textbf{35}, 59--95, 1999.}

\bibitem{Wamba2012} P.~M.~Kouotchop Wamba, A.~Ntyam and J.~Wouafo Kamga.
\newblock{Tangent lift of higher order of multivector fields and applications.}
\newblock{\emph{J. Math. Sci. Adv. Appl.} \textbf{15}, 13--36, 2012.}

\bibitem{Kurose2007}
T.~Kurose.
\newblock{Statistical Manifolds Admitting Torsion (in Japanese).}
\newblock{Geometry and Something. Fukuoka University, 2007.}
	
\bibitem{Lauritzen1987} S.~L.~Lauritzen.
\newblock{Statistical manifolds.}
\newblock{\emph{Diff. Geom. Stat. Inference} \textbf{10}, 163--216, 1987.}

\bibitem{Le2006} H\^ong V\^an L\^e.
\newblock{Statistical manifolds are statistical models.}
\newblock{\emph{J. Geom.} \textbf{84}, 83--93, 2006.}

\bibitem{Mackenzie1992} K.~C.~H.~Mackenzie,
\newblock{Double Lie algebroids and second-order geometry. I,}
\newblock{\emph{Adv. Math.} \textbf{94} (1992), no. 2, 180--239.}

\bibitem{Mackenzie2005} K.~C.~H.~Mackenzie,
\newblock{General theory of Lie groupoids and Lie algebroids,}
\newblock{\emph{London Mathematical Society Lecture Note Series} \textbf{213}, Cambridge University Press, Cambridge, 2005.}

\bibitem{Marmo1985} G.~Marmo, E.~J.~Saletan, A.~Simoni, B.~Vitale.
\newblock{Dynamical Systems: Differential Geometric Approach to Symmetry and Reduction.}
\newblock{J.Whiley, Chichester, 1985}

\bibitem{Matsuzoe2003}
H.~Matsuzoe and J.-i.~Inoguchi.
\newblock{Statistical structures on tangent bundles.}
\newblock{\emph{Appl. Sci.} \textbf{5}, 55--75, 2003.}

\bibitem{Matsuzoe2007}
H.~Matsuzoe.
\newblock{Geometry of statistical manifolds and its generalization.}
\newblock{in: \emph{Topics in contemporary differential geometry, complex analysis and mathematical physics}, 244--251, World Sci. Publ., Hackensack, NJ, 2007.}

\bibitem{Matsuzoe2010}
H.~Matsuzoe.
\newblock{Statistical manifolds and affine differential geometry.}
\newblock{\emph{Probabilistic approach to geometry}, 303--321, Adv. Stud. Pure Math., \textbf{57}, Math. Soc. Japan, Tokyo, 2010.}
	
\bibitem{Matumoto1993}
T.~Matumoto.
\newblock{Any statistical manifold has a contrast function: On the $C^3$-functions
	taking the minimum at the diagonal of the product manifold.}
\newblock{\emph{Hiroshima Math. J.} \textbf{23}, 327--332, 1993.}

\bibitem{Meinrenken2017}
E.~Meinrenken.
\newblock{Lie groupoids and {L}ie algebroids.}
\newblock{{L}ecture notes, 2017, http://www.math.toronto.edu/mein/teaching/lectures.html .}

\bibitem{Morimoto1970} A.~Morimoto,
\newblock{Liftings of tensor fields and connections to tangent bundles of higher order,}
\newblock{\emph{Nagoya Math. J.} \textbf{40} (1970), 99--120.}

\bibitem{Murray1993}
M.~K.~Murray and J.~W.~Rice.
\newblock{Differential geometry and statistics.}
\newblock{\emph{Monographs on Statistics and Applied Probability} \textbf{48},
Chapman \& Hall, London, 1993.}

\bibitem{Peyghan2021}
E.~Peyghan, D.~Seifipour, and A.~M.~Blaga.
\newblock{Geometry of lift metrics and lift connections on the tangent bundle.}
\newblock{\emph{Turkish J. Math.} \textbf{46} (2022), 2335--2352. }

\bibitem{Pradines1974}
J.~Pradines.
\newblock{Fibr\'es vectoriels doubles et calcul des jets non holonomes.}
\newblock{Notes polycopiees. Amiens, 1974.}



\bibitem{Rao1945}
C.~R. Rao.
\newblock{Information and accuracy attainable in the estimation of statistical parameters.}
\newblock{\emph{Bull. Calcutta. Math. Soc.} \textbf{ 37}, 81--91, 1945.}

\bibitem{Sasaki1958}
S.~Sasaki.
\newblock{On the differential geometry of tangent bundles of Riemannian manifolds.}
\newblock{\emph{Tohoku Math J.} \textbf{10}, 338--354, 1958.}

\bibitem{Wheeler1989} J.~A.~Wheeler.
\newblock{Information, Physics, Quantum: the search for links.}
\newblock{Proc. 3rd Int. Symp. Foundations of Quantum Mechanics, Tokyo, pp. 354--368, 1989.}

\bibitem{Tulczyjew1974} W.~M.~Tulczyjew,
\newblock{Hamiltonian Systems, Lagrangian systems and the Legendre transformation,}
\newblock{\emph{Symp. Math.} \textbf{14}, Roma (1974), 247--258.}

\bibitem{Tulczyjew1977} W.~M.~Tulczyjew,
\newblock{The Legendre transformation},
\newblock{\emph{Ann. Inst. H. Poincar\'e, Sect. A}, \textbf{27} (1977), 101--114.}

\bibitem{Tulczyjew1989} W.~M.~Tulczyjew.
\newblock{Geometric Formulation of Physical Theories.}
\newblock{\emph{Monographs and Textbooks in Physical Science} vol. \textbf{11}, Bibliopolis,
Naples, 1989.}

\bibitem{Yano1967} K.~Yano and S.~Ishihara.
\newblock{Horizontal lifts of tensor fields and connections to tangent bundles.}
\newblock{\emph{J. Math. Mech.} \textbf{16} (1967), 1015--1029.}

\bibitem{Yano1973} K.~Yano and S.~Ishihara.
\newblock{Tangent and Cotangent Bundles.}
\newblock{Marcel Dekker, Inc.,1973.}

\bibitem{Zhang2020}
J.~Zhang and G.~Khan.
\newblock{Statistical mirror symmetry.}
\newblock{\emph{Diff. Geom. Appl.} \textbf{73}, 101678, 2020.}

\end{thebibliography}
\end{document}